\documentclass[12pt, reqno]{amsart}
\usepackage{amsmath, amsthm, amscd, amsfonts, amssymb, graphicx, color}
\usepackage[bookmarksnumbered, colorlinks, plainpages]{hyperref}
\hypersetup{colorlinks=true,linkcolor=red, anchorcolor=green, citecolor=cyan, urlcolor=red, filecolor=magenta, pdftoolbar=true}

\newtheorem{theorem}{Theorem}[section]
\newtheorem{lemma}[theorem]{Lemma}

\newtheorem{corollary}[theorem]{Corollary}
\theoremstyle{definition}
\newtheorem{definition}[theorem]{Definition}

\theoremstyle{remark}

\numberwithin{equation}{section}

\begin{document}

\setcounter{page}{1}

\title{The pluriharmonic Hardy space and Toeplitz Operators}

\author[Yuanqi Sang  \MakeLowercase{and} Xuanhao Ding]{Yuanqi Sang \MakeLowercase{and} Xuanhao Ding$^{*}$}

\address{College of Mathematics and Statistics, Chongqing University, Chongqing 401331, P.R. China.}
\email{\textcolor[rgb]{0.00,0.00,0.84}{yqisang@163.com}}

\address{School of Mathematics and Statistics, Chongqing Technology and Business University,
 Chongqing 400067, P.R. China.}
\email{\textcolor[rgb]{0.00,0.00,0.84}{xuanhaod@qq.com}}


\let\thefootnote\relax\footnote{}



\date{\today.
\newline \indent $^{*}$Corresponding author}

\begin{abstract}
Compared with harmonic Bergman spaces, this paper introduces a new function space which is called the pluriharmonic Hardy space $h^{2}(\mathbb{T}^{2})$. We character (semi-) commuting Toeplitz operators on $h^{2}(\mathbb{T}^{2})$ with bounded pluriharmonic symbols. Interestingly, these results are quite different from the corresponding properties of  Toeplitz operators on Hardy spaces, Bergman spaces and harmonic Bergman spaces. Our method for Toeplitz operators on $h^{2}(\mathbb{T}^{2})$ gives new insight into the study of commuting Toeplitz operators on harmonic Bergman spaces.
\end{abstract} \maketitle

\section{Introduction}
Let $\mathbb{D}=\{\xi\in\mathbb{C}:|\xi|<1\}$ be the unit disk in the complex plane $\mathbb{C}$ and  $\mathbb{T}$ be its boundary.
The bidisk $\mathbb{D}^{2}$ and the torus $\mathbb{T}^{2}$ are the subsets of $\mathbb{C}^{2}$   which are Cartesian
 products of two copies $\mathbb{D}$ and $\mathbb{T},$ respectively.
Let $d\sigma$ be the normalized Haar measure on $\mathbb{T}^{2},$
the Hardy space $H^{2}(\mathbb{T}^{2})$ is the
closure of the analytic polynomials in $L^{2}(\mathbb{T}^{2},d\sigma)$.
Let $dA$ denote the normalized area measure on  $\mathbb{D},$
the harmonic Bergman space $b^{2}$ is the closed subspace of the Lebesgue space $L^{2}(\mathbb{D},dA)$
consisting of all harmonic functions on $\mathbb{D}.$
One can check the relation
\begin{equation}{\label{1.9}}
b^{2}(\mathbb{D})=L^{2}_{a}(\mathbb{D})+\overline{L^{2}_{a}(\mathbb{D})},
\end{equation}
where $L^{2}_{a}(\mathbb{D})$ is the closed subspace of the Lebesgue space $L^{2}(\mathbb{D},dA)$
consisting of all holomorphic functions on $\mathbb{D}.$
It is well known that
$H^{2}(\mathbb{T})+\overline{H^{2}(\mathbb{T})}\cong L^{2}(\mathbb{T})$, the projection from  $L^{2}(\mathbb{T})$ onto $H^{2}(\mathbb{T})+\overline{H^{2}(\mathbb{T})}$ is
the identity operator. In this case, Topelitz operators on $H^{2}(\mathbb{T})+\overline{H^{2}(\mathbb{T})}$
are multiplication operators. However,
we have $H^{2}(\mathbb{T}^{2})+\overline{H^{2}(\mathbb{T}^{2})}\subsetneqq L^{2}(\mathbb{T}^{2}).$
For example, $z_{1}\overline{z_{2}}\in L^{2}(\mathbb{T}^{2}),$ but $z_{1}\overline{z_{2}}$ does not belong to
$H^{2}(\mathbb{T}^{2})+\overline{H^{2}(\mathbb{T}^{2})}.$ Inspired by the harmonic Bergman space, we introduce the following definition.
\begin{definition}
We define the pluriharmonic Hardy space $h^{2}(\mathbb{T}^{2})$ by
\begin{equation}{\label{1.2}}
h^{2}(\mathbb{T}^{2})=H^{2}(\mathbb{T}^{2})+\overline{H^{2}(\mathbb{T}^{2})}.
\end{equation}
$h^{2}(\mathbb{T}^{2})$ is a Hilbert with the inner product
\[\langle F,G\rangle=\int_{\mathbb{T}^{2}}F(z)\overline{G(z)}d\sigma(z).\]
It is easy to vertify  $ \{z^{i}_{1}z^{j}_{2}\}_{i,j\geq 0}\bigcup\{\overline{z_{1}}^{k}\overline{z_{2}}^{l}\}_{k,l\geq 0}$ is
an orthonormal basis of $h^{2}(\mathbb{T}^{2})$.
\end{definition}
Let $\partial_{i}$ denote $\frac{\partial}{\partial{\lambda_{i}}}$ and
$\bar{\partial_{i}}$ denote $\frac{\partial}{\partial{\overline{\lambda_{i}}}}.$
Recall that a complex-valued  $f\in C^{2}(\mathbb{D}^{2})$ is said to be pluriharmonic if
\[\partial_{i}\bar{\partial_{j}}f=0,~~~~i,j=1,2.\]
If $f\in h^{2}(\mathbb{T}^{2}),$ the Poisson integral of $f$ is pluriharmonic in $\mathbb{D}^{2}$.
\begin{definition}
For $f \in L^{2}(\mathbb{T}^{2}),$  the Toeplitz operator $\widehat{T}_f$
with symbol $f$ is densely defined  on the pluriharmonic Hardy space $h^{2}(\mathbb{T}^{2})$
by \[\widehat{T}_f h=Q(fh)\] for all polynomials $h,$ where $Q$ is the orthogonal projection from
$L^{2}(\mathbb{T}^{2})$ onto $h^{2}(\mathbb{T}^{2}).$
\end{definition}

On the harmonic Bergman space $b^{2}(\mathbb{D}),$   B. R. Choe and Y. J. Lee \cite{choe1999commuting}
proved that two  Toeplitz operators with holomorphic symbols commute
if and only if a nontrivial linear combination of the symbols is constant,
this result has been extended to various domains such as the polydisk (\cite{choe2006note}) and the unit ball (\cite{lee2002some}).
In \cite{choe1999commuting,choe2004commutants}, B. R. Choe and Y. J. Lee posed two questions.
First, {\em if an analytic Toeplitz operator and
a co-analytic Toeplitz operator commute on $b^{2}(\mathbb{D})$, then is one of their symbols constant ?}
 Second, {\em whether Toeplitz operators with general harmonic symbols
can commute on $b^{2}(\mathbb{D})$ only in the obvious cases ?}  B. R. Choe and Y. J. Lee \cite{choe2004commutants}
showed that the answer to the first question is yes under some additional noncyclictiy hypothesis,
and whether the noncyclicity hypotheses can be removed or not remains open.
We will be concerned with these two similar questions on $h^{2}(\mathbb{T}^{2}).$
\eqref{1.9} and \eqref{1.2} show that  $h^{2}(\mathbb{T}^{2})$ and
 $b^{2}(\mathbb{D})$ have similar structures.
In particular, $L^{2}_{a}$ is isometrically identified with a closed subspace of $H^{2}(\mathbb{T}^{2})$(see \cite{guo2009multiplication}),
by \eqref{1.9}, thus $b^{2}(\mathbb{D})$ is isometrically identified with a closed subspace of $h^{2}(\mathbb{T}^{2}).$
Our results for Toeplitz operators on $h^{2}(\mathbb{T}^{2})$ may offer some insight into the study of similar questions for
Toeplitz operators on $b^{2}(\mathbb{D}).$

On the pluriharmonic Hardy space $h^{2}(\mathbb{T}^{2}),$
 Liu and the second author \cite{liu2013toeplitz} obtained a characterization of (semi-) commuting Toeplitz
operators with holomorphic symbols.
For completeness, we will state these results as follows. In order to state these results, we need some notations.
Let $P$ be the projection from
$L^{2}(\mathbb{T}^{2})$ onto $H^{2}(\mathbb{T}^{2}).$ For each function $f$ in $L^{2}(\mathbb{T}^{2})$, let
\begin{equation*}
\begin{split}
f_{+}&=Pf,\\
f_{-}&=(Q-P)f.
\end{split}
\end{equation*}
Let $z=(z_{1},z_{2})\in \mathbb{T}^{2}.$
\begin{theorem}\label{01}\cite{liu2013toeplitz}
Let $\varphi,\psi$ be two bounded functions in $H^{2}(\mathbb{T}^{2})$. Then  $\widehat{T}_{\varphi}\widehat{T}_{\psi}=\widehat{T}_{\psi}\widehat{T}_{\varphi}$ if and only if

~~~~~~~~~~~~~~~~(i)  ~both $\varphi$ and $\psi$ are functions of one variable $z_{1}$; or

~~~~~~~~~~~~~~~~(ii)  both $\varphi$ and $\psi$ are functions of one variable $z_{2}$; or

~~~~~~~~~~~~~~~~(iii) a nontrivial linear combination of $\varphi$ and $\psi$ is constant.

\end{theorem}
\begin{theorem}\label{02}\cite{liu2013toeplitz}
Let $\varphi,\psi$ be two bounded functions in $H^{2}(\mathbb{T}^{2})$. Then  $\widehat{T}_{\varphi}\widehat{T}_{\psi}=\widehat{T}_{\varphi \psi}$ if and only if

~~~~~~~~~~~~~~~~(i)  ~~~~both $\varphi$ and $\psi$ are functions of one variable $z_{1}$; or

~~~~~~~~~~~~~~~~(ii)  ~~both $\varphi$ and $\psi$ are functions of one variable $z_{2}$; or

~~~~~~~~~~~~~~~~(iii)  either $\varphi$ or $\psi$ is constant.

\end{theorem}
In \cite{Sang2016toeplitz}, the authors obtained a necessary and sufficient condition for an analytic Toeplitz operator that
commutes with another  co-analytic Toeplitz operator  on $h^{2}(\mathbb{T}^{2}).$
\begin{theorem}{\label{1.5}}\cite{Sang2016toeplitz}
Let $\varphi,\psi$ be two bounded functions in $H^{2}(\mathbb{T}^{2})$.
Then $\widehat{T}_{\varphi}\widehat{T}_{\bar{\psi}}=\widehat{T}_{\bar{\psi}}\widehat{T}_{\varphi}$ if and only if
$\widehat{T}_{\varphi}\widehat{T}_{\bar{\psi}}=\widehat{T}_{\varphi\bar{\psi}}$  if and only if

~~~~~~~~~(a) either $\varphi$ or $\psi$ is constant; or

~~~~~~~~~(b) $\varphi$  is a function of one variable $z_{1}$ and $\psi$ is a function of one variable $z_{2};$ or

~~~~~~~~~(c) $\varphi$  is a function of one variable $z_{2}$ and $\psi$ is a function of one variable $z_{1}.$
\end{theorem}
In this paper, we consider the problem of when two Toeplitz operaotrs with
bounded pluriharmonic symbols commute or semi-commute.
Our main results are as follows.
\begin{theorem}\label{27}
Let $f,g$ be two bounded functions in $h^{2}(\mathbb{T}^{2})$.
Then $\widehat{T}_f\widehat{T}_g=\widehat{T}_{fg}$ if and only if one of the following conditions is satisfied:

(A)   $f_{+}$ and $g_{+}$ are the functions of one variable $z_{1}$, and $f_{-}$ and $g_{-}$ are the functions of one variable $z_{2};$ or

(B)   $f_{+}$ and $g_{+}$ are the functions of one variable $z_{2}$, and $f_{-}$ and $g_{-}$ are the functions of one variable $z_{1};$ or

(C)   Either $f$ or $g$  is constant.
\end{theorem}

\begin{theorem}\label{26}
Let $f,g$  be two bounded functions in  $h^{2}(\mathbb{T}^{2})$.
Then $\widehat{T}_f\widehat{T}_g=\widehat{T}_g\widehat{T}_f$ if and only if one of the following conditions is satisfied:

(I)   $f_{+}$ and $g_{+}$ are the functions of one variable $z_{1}$, and $f_{-}$ and $g_{-}$ are the functions of one variable $z_{2};$

(II)   $f_{+}$ and $g_{+}$ are the functions of one variable $z_{2}$, and $f_{-}$ and $g_{-}$ are the functions of one variable $z_{1};$

(III)  A nontrivial linear combination of $f$ and $g$ is constant.
\end{theorem}

This paper is organized in the following way. In section 2, we present some preliminaries.
By making use of matrix representation,
we give some new characterization
of Theorem \ref{01}, Theorem \ref{02} and Theorem \ref{1.5} in section 3.
Using these results, we  completely
characterize (semi-)commuting Toeplitz operators with bounded pluriharmonic harmonic
symbols in section 4 and section 5, respectively.
\section{Preliminaries}

\subsection{Integral representation}
\hspace*{\parindent}
Recall that the  Hardy space $H^{2}(\mathbb{T})$ is a reproducing kernel Hilbert space, with the kernel
\[K_{{\lambda}_{1}}(\omega_{1})=\frac{1}{1-\overline{{\lambda}_{1}}\omega_{1}},~~\lambda_{1},\omega_{1}\in\mathbb{D},\]
so that
\begin{equation}\label{28}
\begin{split}
h({\lambda}_{1})=\langle~ h,K_{{\lambda}_{1}}\rangle
\end{split}
\end{equation}
for every $h\in H^{2}(\mathbb{D}).$
We will call $k_{{\lambda}_{1}}(\omega_{1})=K_{{\lambda}_{1}}/\|K_{{\lambda}_{1}}\|
=\frac{({1-|{\lambda}_{1}|^{2}})^{\frac{1}{2}}}{1-\overline{{\lambda}_{1}}\omega_{1}}$ the
normalized reproducing kernel.
Clearly, the reproducing kernel of $H^{2}(\mathbb{D}^{2})$ at the point $\lambda$ with
coordinates $({\lambda}_{1},{\lambda}_{2})$ in $\mathbb{D}^{2}$ is given by
\[\mathbf{K}_{{\lambda}}(\omega)=\prod^{2}_{i=1}K_{{\lambda}_{i}}(\omega_{i}).\]
Thus the normalized reproducing kernel $\mathbf{k}_{{\lambda}}$ of $H^{2}(\mathbb{D}^{2})$
is in the form \[\mathbf{k}_{{\lambda}}(\omega)=\prod^{2}_{i=1}k_{{\lambda}_{i}}(\omega_{i}).\]

If $f\in L^{1}(\mathbb{T}^{2}),$
the Poisson integral of $f$ is given by
\begin{equation*}
\begin{split}
\mathcal{P}[f](\lambda)&=\int_{\mathbb{T}^{2}}f(z)\prod^{2}_{i=1}\frac{1-|\lambda_{i}|^{2}}{|1-\lambda_{i}\bar{z_{i}}|^{2}}d\sigma(z)\\
&=\int_{\mathbb{T}^{2}}f(z)|\mathbf{k}_{{\lambda}}(z)|^{2}d\sigma(z)\\
&=\langle f\mathbf{k}_{{\lambda}},\mathbf{k}_{{\lambda}}\rangle.
\end{split}
\end{equation*}
Since $\lim_{r\to1}\mathcal{P}[f](rz)=f(z)$ for almost every $z \in\mathbb{T}^{2}$ \cite[Theorem 2.3.1]{rudin1969function}, we can identify functions in  $h^{2}(\mathbb{T}^{2})$
with its Poisson integral.

For any bounded linear operator $S$ on $H^{2}(\mathbb{T}^{2})$, the Berezin transform of $S$ is the function $\widetilde{S}$ on
$\mathbb{D}^{2}$ defined by
\begin{equation}\label{52}
\begin{split}
\widetilde{S}(\lambda)&=\langle S\mathbf{k}_{{\lambda}},\mathbf{k}_{{\lambda}}\rangle\\
&=\int_{\mathbb{T}^{2}}S\mathbf{k}_{{\lambda}}(z)\overline{\mathbf{k}_{{\lambda}}(z)}d\sigma(z).
\end{split}
\end{equation}

For $f\in L^{2}(\mathbb{T}^{2}),$ the Hardy projection $P$ has the integral representation
\begin{equation}\label{30}
\begin{split}
Pf(\lambda)=&\int_{\mathbb{T}^{2}}f(z)\overline{\mathbf{K}_{{\lambda}}(z)}d\sigma(z).\\
\end{split}
\end{equation}
Similarly, let $P^{-}$ be the orthogonal projection from $L^{2}(\mathbb{T}^{2})$
onto $\overline{H^{2}(\mathbb{T}^{2})},$
we can write the projection $P^{-}$ as the integral operator
\begin{equation}\label{31}
\begin{split}
P^{-}f(\lambda)=&\int_{\mathbb{T}^{2}}f(z)\mathbf{K}_{{\lambda}}(z)d\sigma(z).\\
\end{split}
\end{equation}
The above integral formulas for $P$ and $P^{-}$ show that  $P$ and $P^{-}$ can be
extend to $L^{1}(\mathbb{T}^{2}).$
By \eqref{30} and \eqref{31}, we have
\begin{equation}\label{56}
\begin{split}
P(f)(0)=\int_{\mathbb{T}^{2}}f(z)d\sigma(z)=P^{-}(f)(0).
\end{split}
\end{equation}
In addition, the reproducing kernel $\mathbf{K}_{\lambda}$ has the following nice property \cite[Theorem 1.3]{ding2003products}:
\begin{equation}\label{54}
\begin{split}
P(\bar{\varphi}\mathbf{K}_{\lambda})=\overline{\varphi(\lambda)}\mathbf{K}_{\lambda}
\end{split}
\end{equation}
if $\varphi\in H^{2}(\mathbb{T}^{2}).$

Each point evaluation is easily verified to be a bounded linear
functional on $h^{2}(\mathbb{T}^{2})$, the Riesz representation theorem tells us that there exists a unique function
$\mathbf{R}_{{\lambda}}$ in $h^{2}(\mathbb{T}^{2})$ such that
\begin{equation}\label{29}
\begin{split}
g(\lambda)=\langle g,\mathbf{R}_{{\lambda}}\rangle=\int_{\mathbb{T}^{2}}g(z)\overline{\mathbf{R}_{{\lambda}}(z)}d\sigma(z)
\end{split}
\end{equation}
for all $g$ in $h^{2}(\mathbb{T}^{2})$.
Since \eqref{1.2}, there is a simple relation between  $\mathbf{R}_{{\lambda}}$ and $\mathbf{K}_{{\lambda}}:$
\begin{equation}
\begin{split}
\mathbf{R}_{{\lambda}}=\mathbf{K}_{{\lambda}}+\overline{\mathbf{K}_{{\lambda}}}-1.
\end{split}
\end{equation}
Thus each $\mathbf{R}_{{\lambda}}$ is real-valued and the formula \eqref{29}
leads us to the following integral representation of the projection $Q$:
\begin{equation}\label{45}
\begin{split}
Qf(\lambda)=&\int_{\mathbb{T}^{2}}f(z)\mathbf{R}_{{\lambda}}(z)d\sigma(z),\\
\end{split}
\end{equation}
we see  from \eqref{45} that the projection $Q$ can be rewritten as
\begin{equation*}
\begin{split}
Qf=P(f)+P^{-}(f)-P(f)(0)
\end{split}
\end{equation*}
for $f\in L^{2}(\mathbb{T}^{2}).$

For $f\in L^{1}(\mathbb{T}^{2})$ and fixed ${\lambda}_{i}\in \mathbb{D}~ (1\leq i\leq 2),$ we define
\begin{equation*}
\begin{split}
(P_{i}f)|_{z_{i}={\lambda}_{i}}=\int_{\mathbb{T}}f(z)\overline{K_{{\lambda}_{i}}(z_{i})}d\sigma(z_{i})
\end{split}
\end{equation*}
and
\begin{equation*}
\begin{split}
(P^{-}_{i}f)|_{z_{i}={\lambda}_{i}}=\int_{\mathbb{T}}f(z)K_{{\lambda}_{i}}(z_{i})d\sigma(z_{i}).
\end{split}
\end{equation*}

Using the boundedness of $P_{i}$ and $P^{-}_{i}$ on $L^{2}(\mathbb{T}^{2})$, one can easily verify the following facts:

 $\bullet$  $ P_{1}$ commutes with $P_{2}$, and $P=P_{1}P_{2}.$

 $\bullet$ $P^{-}_{1}$ commutes with $P^{-}_{2}$, and $P^{-}=P^{-}_{1}P^{-}_{2}.$

\subsection{Matrix representation}
\hspace*{\parindent}
For $f \in L^{2}(\mathbb{T}^{2}),$  the Toeplitz operator ${T}_f$ and the small Hankel operator $\Gamma_f$
with symbol $f$ are densely defined  on the Hardy space $H^{2}(\mathbb{T}^{2})$
by \[{T}_f h=P(fh)~~~~~~\text{and}~~~~~~~\Gamma_f h=(Q-P)(fh)\]
 for all polynomials $h.$

For $\phi\in L^{2}(\mathbb{T}^{2})$, under the decomposition $h^{2}(\mathbb{T}^{2})= H^{2}(\mathbb{T}^{2})\bigoplus\overline{{H^{2}_{0}(\mathbb{T}^{2})}}$,
where ${{H^{2}_{0}(\mathbb{T}^{2})}}=\Big\{f\in{{H^{2}(\mathbb{T}^{2})}}:f(0)=0\Big\},$
the Toeplitz operator $\widehat{T}_{\phi}$ has the following operator matrix representation:
\begin{gather}\label{36}
\widehat{T}_{\phi}=
\begin{bmatrix}T_\phi& {\Gamma^{*}_{\bar{\phi}}}\\
   \Gamma_\phi &  S_\phi\end{bmatrix},\quad
\end{gather}
where the operator $S_\phi$ is an operator on $\overline{{H^{2}_{0}(\mathbb{T}^{2})}}$ defined by
\begin{equation}
\begin{split}
S_\phi \bar{v}=(Q-P)(\phi \bar{v})
\end{split}
\end{equation}
for every polynomial $v\in {H^{2}_{0}(\mathbb{T}^{2})}.$
If $\phi\in L^{2}(\mathbb{T}^{2}),$ we define ${\Gamma^{*}_{\bar{\phi}}}\bar{v}$ by
\[{\Gamma^{*}_{\bar{\phi}}}\bar{v}(\lambda)
=P(\phi \bar{v})(\lambda)
=\int_{\mathbb{T}^{2}}\phi \bar{v} \overline{\mathbf{K}_{{\lambda}}}d\sigma.\]
Note that the star need no longer be the adjoint (but would of course coincide with the adjoint in case the operator
${\Gamma_{\bar{\phi}}}$ is itself bounded).

The following lemma shows relation between Toeplitz operators $T_{\varphi}$ and Hankel operators $\Gamma_\varphi$.
It is quite useful in studying products of Toeplitz operators.

\begin{lemma}
Let $f,g$ and $h$ be three functions in $ L^{\infty}(\mathbb{T}^{2})$. If $\widehat{T}_{f}\widehat{T}_{g}=\widehat{T}_{h},$ then
the following identity holds.
\begin{equation}\label{24}
\begin{split}
T_{h}=T_{f}T_{g}+{\Gamma^{*}_{\bar{f}}}{\Gamma_{g}}.
\end{split}
\end{equation}
\end{lemma}
\begin{proof}To prove the above identity we use the matrix representations of operators
$\widehat{T}_{f},\widehat{T}_{g}$ and $\widehat{T}_{fg}$ under the decomposition $h^{2}(\mathbb{T}^{2})= H^{2}(\mathbb{T}^{2})\bigoplus\overline{{H^{2}_{0}(\mathbb{T}^{2})}}$.
 Since $\widehat{T}_{f}\widehat{T}_{g}=\widehat{T}_{fg}$, computing the product of the matrices of $\widehat{T}_{f}$ and $\widehat{T}_{g}$ gives
\begin{gather}\label{37}
\widehat{T}_{f}\widehat{T}_{g}=\begin{bmatrix}T_{f}T_{g}+{\Gamma^{*}_{\bar{f}}}{\Gamma_{g}} & T_{f}{\Gamma^{*}_{\bar{g}}}+{\Gamma^{*}_{\bar{f}}}{S_{g}}\\
\Gamma_{f}T_{g}+S_{f}{\Gamma_{g}} & \Gamma_{f}{\Gamma^{*}_{\bar{g}}}+S_{f}S_{g}\end{bmatrix}.\quad
\end{gather}
On the other hand,
\begin{gather*}
\widehat{T}_{h}=\begin{bmatrix}T_{h}&{\Gamma^{*}_{\overline{h}}} \\ \Gamma_{h} & S_{h}\end{bmatrix}.\quad
\end{gather*}
Comparison of the two matrix representations of $\widehat{T}_{h}$ gives \eqref{24}.
\end{proof}

\begin{theorem}
If $\bar{f}$ or $g$ in $H^{\infty}(\mathbb{T}^{2})$, and  $\widehat{T}_{f}\widehat{T}_{g}=\widehat{T}_{h},$ then $fg=h.$
\end{theorem}

\begin{proof}
 Since the small Hankel operator with an analytic symbol is the zero operator, \eqref{24} implies $T_{h}=T_{f}T_{g}.$
By \cite[Theorem 3.1]{ding2003products}, we have that $fg=h.$
\end{proof}
\noindent{\bf{Remark.}}\,\,\,It is well-known that $T_{\varphi}$ satisfies the following characteristic relation
\[T^{*}_{z_{i}}T_{\varphi}T_{z_{i}}=T_{\varphi}\]for $\varphi\in L^{\infty}(\mathbb{T}^{2}),1\leq i\leq 2.$
But this relation does not extend to the Toeplitz operator $\widehat{T}_{\varphi}$ as the following example shows.
Since $\widehat{T}^{*}_{z_{1}}\widehat{T}_{z_{1}}\widehat{T}_{z_{1}}\bar{z_{1}}\bar{z_{2}}=0$ and $\widehat{T}_{z_{1}}\bar{z_{1}}\bar{z_{2}}=\bar{z_{2}}$,
$\widehat{T}^{*}_{z_{1}}\widehat{T}_{z_{1}}\widehat{T}_{z_{1}}\neq \widehat{T}_{z_{1}}.$

Although our main concern is with bounded operators, we will need to make use of
densely defined unbounded operators. Let $H$ be a Hilbert space. Suppose $A_{1}$ and $A_{2}$ are densely defined operators
in $H$, we say that $A_{1} = A_{2}$ if
\[A_{1}p = A_{2}p\]
for each $p\in dom A_{1}\bigcap dom A_{2}.$
 Let $x$ and $y$ be two nonzero vectors in $H,$ ~$x \otimes y$ is the operator of rank one defined by
\[(x \otimes y)f=\left< f,y\right>x\]
for $f\in H.$

\section{Some Lemmas}
\hspace*{\parindent}
Let $h\in L^{2}(\mathbb{T}^{2}),$ $h=h(z_{1})$ denotes $h$ is a function of one variable $z_{1}.$ Similarly,
 $h=h(z_{2})$ denotes $h$ is a function of one variable $z_{2}.$ We denote the semicommutator and commutator
of two Toeplitz operators $\widehat{T}_{f}$ and $\widehat{T}_{g}$ by
\[(\widehat{T}_{f},\widehat{T}_{g}]=\widehat{T}_{fg}-\widehat{T}_{f}\widehat{T}_{g}\]
and
\[[\widehat{T}_{f},\widehat{T}_{g}]=\widehat{T}_{f}\widehat{T}_{g}-\widehat{T}_{g}\widehat{T}_{f}.\]

The following lemma will be used
in the proof of Theorem \ref{26}.
\begin{lemma}\label{1}
Let $\varphi,\psi\in H^{2}(\mathbb{T})$, $\varphi=\varphi(z_{1}),\psi=\psi(z_{1})$.Then
\[\int_{\mathbb{T}^{2}}([\widehat{T}_{\varphi},\widehat{T}_{\bar{\psi}}]k_{{\lambda}_{1}}k_{{\lambda}_{2}})\overline{k_{{\lambda}_{1}}k_{{\lambda}_{2}}}d\sigma
=|{\lambda}_{2}|^{2}\bigg(\varphi({\lambda}_{1})\overline{\psi({\lambda}_{1})}-
\mathcal{P}[\varphi\bar{\psi}]({\lambda}_{1})\bigg)\]
for all $(\lambda_{1},\lambda_{2})\in\mathbb{D}^{2}.$
\end{lemma}
\begin{proof}Since $\bar{\psi}k_{{\lambda}_{1}}\in L^{2}(\mathbb{T}),$ we have
\begin{equation*}
\begin{split}
\bar{\psi}k_{{\lambda}_{1}}
&=P_{1}(\bar{\psi}k_{{\lambda}_{1}})+P^{-}_{1}(\bar{\psi}k_{{\lambda}_{1}})-P_{1}(\bar{\psi}k_{{\lambda}_{1}})(0)\\
&=\bar{\psi}({\lambda}_{1})k_{{\lambda}_{1}}+P^{-}_{1}(\bar{\psi}k_{{\lambda}_{1}})
-\bar{\psi}({\lambda}_{1})\sqrt{1-|{\lambda}_{1}|^{2}}.\\
\end{split}
\end{equation*}
Hence
\begin{equation}\label{53}
\begin{split}
P^{-}_{1}(\bar{\psi}k_{{\lambda}_{1}})=\bar{\psi}k_{{\lambda}_{1}}-\bar{\psi}({\lambda}_{1})k_{{\lambda}_{1}}
+\bar{\psi}({\lambda}_{1})\sqrt{1-|{\lambda}_{1}|^{2}}.
\end{split}
\end{equation}

An easy computation gives
\begin{equation}\label{55}
\begin{split}
\widehat{T}_{\bar{\psi}}k_{{\lambda}_{1}}k_{{\lambda}_{2}}
=&~P(\bar{\psi}k_{{\lambda}_{1}}k_{{\lambda}_{2}})+P^{-}(\bar{\psi}k_{{\lambda}_{1}}k_{{\lambda}_{2}})
-P(\bar{\psi}k_{{\lambda}_{1}}k_{{\lambda}_{2}})(0)\\
=&~P_{1}(\bar{\psi}k_{{\lambda}_{1}})P_{2}(k_{{\lambda}_{2}})+P_{1}^{-}(\bar{\psi}k_{{\lambda}_{1}})P_{2}^{-}(k_{{\lambda}_{2}})
-P(\bar{\psi}k_{{\lambda}_{1}}k_{{\lambda}_{2}})(0)\\
=&~\bar{\psi}({\lambda}_{1})k_{{\lambda}_{1}}k_{{\lambda}_{2}}+P^{-}_{1}(\bar{\psi}k_{{\lambda}_{1}})\sqrt{1-|{\lambda}_{2}|^{2}}
-\bar{\psi}({\lambda}_{1})\sqrt{1-|{\lambda}_{1}|^{2}}\sqrt{1-|{\lambda}_{2}|^{2}}\\
=&~\bar{\psi}({\lambda}_{1})k_{{\lambda}_{1}}k_{{\lambda}_{2}}+\bigg(\bar{\psi}k_{{\lambda}_{1}}-\bar{\psi}({\lambda}_{1})k_{{\lambda}_{1}}
+\bar{\psi}({\lambda}_{1})\sqrt{1-|{\lambda}_{1}|^{2}}\bigg)\sqrt{1-|{\lambda}_{2}|^{2}}\\
&-\bar{\psi}({\lambda}_{1})\sqrt{1-|{\lambda}_{1}|^{2}}\sqrt{1-|{\lambda}_{2}|^{2}}\\
=&~\bar{\psi}({\lambda}_{1})k_{{\lambda}_{1}}k_{{\lambda}_{2}}+
\bigg(\bar{\psi}-\bar{\psi}({\lambda}_{1})\bigg)k_{{\lambda}_{1}}\sqrt{1-|{\lambda}_{2}|^{2}}.
\end{split}
\end{equation}
The second equality follows from that $P=P_{1}P_{2}=P_{2}P_{1}$ and $P^{-}=P^{-}_{1}P^{-}_{2}=P^{-}_{2}P^{-}_{1},$
the third equality follows from \eqref{54},
the fourth equality follows from \eqref{53}.
Thus
\begin{equation}\label{46}
\begin{split}
\int_{\mathbb{T}^{2}}&(\widehat{T}_{\varphi}\widehat{T}_{\bar{\psi}}k_{{\lambda}_{1}}k_{{\lambda}_{2}})\overline{k_{{\lambda}_{1}}k_{{\lambda}_{2}}}d\sigma\\
=&\int_{\mathbb{T}^{2}}(Q{\varphi}\widehat{T}_{\bar{\psi}}k_{{\lambda}_{1}}k_{{\lambda}_{2}})\overline{k_{{\lambda}_{1}}k_{{\lambda}_{2}}}d\sigma\\
=&\int_{\mathbb{T}^{2}}\int_{\mathbb{T}^{2}}\varphi(z_{1})(\widehat{T}_{\bar{\psi}}k_{{\lambda}_{1}}k_{{\lambda}_{2}})(z)
\mathbf{R}_{\eta}(z)d\sigma(z)\overline{k_{{\lambda}_{1}}(\eta_{1})k_{{\lambda}_{2}}(\eta_{2})}
d\sigma(\eta)\\
=&\int_{\mathbb{T}^{2}}\varphi(z_{1})(\widehat{T}_{\bar{\psi}}k_{{\lambda}_{1}}k_{{\lambda}_{2}})(z)\int_{\mathbb{T}^{2}}
\mathbf{R}_{z}(\eta)\overline{k_{{\lambda}_{1}}(\eta_{1})k_{{\lambda}_{2}}(\eta_{2})}
d\sigma(\eta)d\sigma(z)\\
=&\int_{\mathbb{T}^{2}}{\varphi}(\widehat{T}_{\bar{\psi}}k_{{\lambda}_{1}}k_{{\lambda}_{2}})\overline{k_{{\lambda}_{1}}k_{{\lambda}_{2}}}d\sigma(z)\\
=&\left<\widehat{T}_{\bar{\psi}}k_{{\lambda}_{1}}k_{{\lambda}_{2}},\bar{\varphi}k_{{\lambda}_{1}}k_{{\lambda}_{2}}\right>\\ =&\left<\widehat{T}_{\bar{\psi}}k_{{\lambda}_{1}}k_{{\lambda}_{2}},\widehat{T}_{\bar{\varphi}}k_{{\lambda}_{1}}k_{{\lambda}_{2}}\right>\\
=&\left<\bar{\psi}({\lambda}_{1})k_{{\lambda}_{1}}k_{{\lambda}_{2}},\bar{\varphi}({\lambda}_{1})k_{{\lambda}_{1}}k_{{\lambda}_{2}}\right>\\
&+\left<\bigg(\bar{\psi}-\bar{\psi}({\lambda}_{1})\bigg)k_{{\lambda}_{1}}\sqrt{1-|{\lambda}_{2}|^{2}}
,\bigg(\bar{\varphi}-\bar{\varphi}({\lambda}_{1})\bigg)k_{{\lambda}_{1}}\sqrt{1-|{\lambda}_{2}|^{2}}\right>\\
=&~\varphi({\lambda}_{1})\bar{\psi}({\lambda}_{1})+(1-|{\lambda}_{2}|^{2})\bigg(\left<\bar{\psi}k_{{\lambda}_{1}},
\bar{\varphi}k_{{\lambda}_{1}}\right>-\left<\bar{\psi}k_{{\lambda}_{1}},\bar{\varphi}({\lambda}_{1})k_{{\lambda}_{1}}\right>\bigg)\\
=&~\varphi({\lambda}_{1})\bar{\psi}({\lambda}_{1})+(1-|{\lambda}_{2}|^{2})\bigg(
\mathcal{P}[\varphi\bar{\psi}]({\lambda}_{1})-\varphi({\lambda}_{1})\bar{\psi}({\lambda}_{1})\bigg)\\
=&~|{\lambda}_{2}|^{2}\varphi({\lambda}_{1})\bar{\psi}({\lambda}_{1})+(1-|{\lambda}_{2}|^{2})
\mathcal{P}[\varphi\bar{\psi}]({\lambda}_{1}).\\
\end{split}
\end{equation}
The second equality follows from \eqref{45},
the seventh equality follows from \eqref{55}.

On the other hand, we get
\begin{equation}\label{47}
\begin{split}
\int_{\mathbb{T}^{2}}&(\widehat{T}_{\bar{\psi}}\widehat{T}_{\varphi}k_{{\lambda}_{1}}k_{{\lambda}_{2}})\overline{k_{{\lambda}_{1}}k_{{\lambda}_{2}}}d\sigma\\
=&\int_{\mathbb{T}^{2}}(Q\bar{\psi}Q{\varphi}k_{{\lambda}_{1}}k_{{\lambda}_{2}})\overline{k_{{\lambda}_{1}}k_{{\lambda}_{2}}}d\sigma\\
=&\int_{\mathbb{T}^{2}}(Q\bar{\psi}{\varphi}k_{{\lambda}_{1}}k_{{\lambda}_{2}})\overline{k_{{\lambda}_{1}}k_{{\lambda}_{2}}}d\sigma\\
=&\int_{\mathbb{T}^{2}}\int_{\mathbb{T}^{2}}(\bar{\psi}{\varphi}k_{{\lambda}_{1}}k_{{\lambda}_{2}})(z)
\mathbf{R}_{\eta}(z)d\sigma(z)\overline{k_{{\lambda}_{1}}(\eta_{1})k_{{\lambda}_{2}}(\eta_{2})}
d\sigma(\eta_{1},\eta_{2})\\
=&\int_{\mathbb{T}^{2}}(\bar{\psi}{\varphi}k_{{\lambda}_{1}}k_{{\lambda}_{2}})(z)\int_{\mathbb{T}^{2}}
\mathbf{R}_{z}(\eta)\overline{k_{{\lambda}_{1}}(\eta_{1})k_{{\lambda}_{2}}(\eta_{2})}
d\sigma(\eta)d\sigma(z)\\
=&\int_{\mathbb{T}^{2}}(\bar{\psi}{\varphi}k_{{\lambda}_{1}}k_{{\lambda}_{2}})\overline{k_{{\lambda}_{1}}k_{{\lambda}_{2}}}d\sigma\\
=&\left<\bar{\psi}{\varphi}k_{{\lambda}_{1}}k_{{\lambda}_{2}},
k_{{\lambda}_{1}}k_{{\lambda}_{2}}\right>=\mathcal{P}[\varphi\bar{\psi}]({\lambda}_{1}).\\
\end{split}
\end{equation}
The third equality follows from \eqref{45}.

So combining \eqref{46} with \eqref{47}, we have
\begin{equation*}
\begin{split}
\int_{\mathbb{T}^{2}}([\widehat{T}_{\varphi},\widehat{T}_{\bar{\psi}}]k_{{\lambda}_{1}}k_{{\lambda}_{2}})\overline{k_{{\lambda}_{1}}k_{{\lambda}_{2}}}d\sigma
=|{\lambda}_{2}|^{2}\bigg(\varphi({\lambda}_{1})\bar{\psi}({\lambda}_{1})-
\mathcal{P}[\varphi\bar{\psi}]({\lambda}_{1})\bigg).
\end{split}
\end{equation*}
\end{proof}

\noindent{\bf{Remark.}}\,\,\,Although $\widehat{T}_{\varphi}\widehat{T}_{\bar{\psi}}$
and $\widehat{T}_{\bar{\psi}}\widehat{T}_{\varphi}$
may not be bounded, the integral formula \eqref{52} of their Berezin transform  still make sense.

In the following four lemmas, $u$ denotes a polynomial in $ H^{2}(\mathbb{T}^{2}),$
$v$ denotes a polynomial in $ H^{2}_{0}(\mathbb{T}^{2}).$
\begin{lemma}\label{4}
If $\varphi$ and $\psi$ are functions in $ H^{2}(\mathbb{T}^{2}),$ the following equalities hold.
\begin{equation*}
\begin{split}
S_{\varphi}{\Gamma_{\bar{\psi}}}&=\Gamma_{\bar{\psi}}T_{\varphi}=\Gamma_{\varphi\bar{\psi}},\\
T_{\bar{\varphi}}{\Gamma^{*}_{\bar{\psi}}}&={\Gamma^{*}_{\bar{\psi}}}{S_{\bar{\varphi}}}=\Gamma^{*}_{\varphi\bar{\psi}}.
\end{split}
\end{equation*}

\end{lemma}

\begin{proof}
An easy computation gives
\begin{equation*}
\begin{split}
\Gamma_{\bar{\psi}}T_{\varphi}u=&(Q-P)\bar{\psi}P\varphi u\\
=&(Q-P)\bar{\psi}\varphi u=\Gamma_{\varphi\bar{\psi}}u.
\end{split}
\end{equation*}
The third equality follows from the fact that $\varphi u\in H^{2}(\mathbb{T}^{2}).$

Using the integral representations of $P$ (see \eqref{30}) and $Q$ (see \eqref{45}), we have
\begin{equation*}
\begin{split}
(S_{\varphi}{\Gamma_{\bar{\psi}}}u)(z)&=(Q-P)(\varphi\Gamma_{\bar{\psi}}u)(z)\\
&=\int_{\mathbb{T}^{2}}\varphi\Gamma_{\bar{\psi}}u(\mathbf{K}_{z}-1)d\sigma\\
&=\left< {\Gamma_{\bar{\psi}}}u,\overline{\varphi(\mathbf{K}_{z}-1)}\right >\\
&=\left< (Q-P)(\bar{\psi}u),\overline{\varphi(\mathbf{K}_{z}-1)}\right >\\
&=\left< \bar{\psi}u,\overline{\varphi(\mathbf{K}_{z}-1)}\right >\\
&=\int_{\mathbb{T}^{2}}\varphi{\bar{\psi}}u(\mathbf{K}_{z}-1)d\sigma\\
&=(Q-P)(\varphi\bar{\psi}u)(z)=(\Gamma_{\varphi\bar{\psi}}u)(z).\\
\end{split}
\end{equation*}
The fifth equality follows from the fact that $\overline{\varphi(\mathbf{K}_{z}-1)}\in \overline{H^{2}_{0}(\mathbb{T}^{2})}.$
An easy computation gives
\begin{equation*}
\begin{split}
{\Gamma^{*}_{\bar{\psi}}}{S_{\bar{\varphi}}}\bar{v}&=P\psi(Q-P)\overline{\varphi v}\\
&=P(\psi\overline{\varphi v})=\Gamma^{*}_{\varphi\bar{\psi}}\bar{v}.
\end{split}
\end{equation*}
The third equality follows from the fact that $\overline{\varphi v}\in \overline{H^{2}_{0}(\mathbb{T}^{2})}.$

Using the integral representation of $P$ (see \eqref{30}), we get
\begin{equation*}
\begin{split}
(T_{\bar{\varphi}}{\Gamma^{*}_{\bar{\psi}}}\bar{v})(z)&=P(\bar{\varphi}\Gamma^{*}_{\bar{\psi}}\bar{v})(z)\\
&=\int_{\mathbb{T}^{2}}\bar{\varphi}\Gamma^{*}_{\bar{\psi}}\bar{v}\overline{\mathbf{K}_{z}}d\sigma\\
&=\left< \Gamma^{*}_{\bar{\psi}}\bar{v},\varphi\mathbf{K}_{z}\right >\\
&=\left< P(\psi\bar{v}),\varphi\mathbf{K}_{z}\right >\\
&=\left< \psi\bar{v},\varphi\mathbf{K}_{z}\right >\\
&=\int_{\mathbb{T}^{2}}\bar{\varphi}\psi\bar{v}\overline{\mathbf{K}_{z}}d\sigma\\
&=P(\bar{\varphi}\psi\bar{v})(z)=(\Gamma^{*}_{\varphi\bar{\psi}}\bar{v})(z).\\
\end{split}
\end{equation*}
The fifth equality follows from the fact that $\varphi\mathbf{K}_{z}\in H^{2}(\mathbb{T}^{2}).$
\end{proof}
The following three lemmas play an important role in understanding (semi)commutativity of Toeplitz operators on
$h^{2}(\mathbb{T}^{2}).$
\begin{lemma}\label{2}
Let $\varphi,\psi\in H^{2}(\mathbb{T}^{2})$. The following statements are equivalent.

~~~~~~~~~~~~(a) $\widehat{T}_{\varphi}\widehat{T}_{\psi}=\widehat{T}_{\varphi \psi}.$

~~~~~~~~~~~~(b) $T_{\varphi}{\Gamma^{*}_{\bar{\psi}}}+{\Gamma^{*}_{\bar{\varphi}}}{S_{\psi}}=\Gamma^{*}_{\overline{\varphi \psi}}.$

~~~~~~~~~~~~(c) $\widehat{T}_{\bar{\psi}}\widehat{T}_{\bar{\varphi}}=\widehat{T}_{\overline{\varphi \psi}}.$

~~~~~~~~~~~~(d) ${\Gamma_{\bar{\psi}}}T_{\bar{\varphi}}+{S_{\bar{\psi}}}{\Gamma_{\bar{\varphi}}}=\Gamma_{\overline{\varphi \psi}}.$

~~~~~~~~~~~~(e) ~(e1) Both $\varphi$ and $\psi$ are functions of one variable $z_{1};$ or

~~~~~~~~~~~~~~~~~(e2)  Both $\varphi$ and $\psi$ are functions of one variable $z_{2};$ or

~~~~~~~~~~~~~~~~~(e3) Either $\varphi$ or $\psi$ is constant.
\end{lemma}
\begin{proof}
$(a)\Leftrightarrow(b)$: We use  the matrix representations of $\widehat{T}_{\varphi}\widehat{T}_{\psi}$ (see \eqref{37})
and $\widehat{T}_{\varphi \psi}$ (see \eqref{36}) to obtain
\begin{gather*}
\widehat{T}_{\varphi}\widehat{T}_{\psi}=\begin{bmatrix}T_{\varphi}T_{\psi} & T_{\varphi}{\Gamma^{*}_{\bar{\psi}}}+{\Gamma^{*}_{\bar{\varphi}}}{S_{\psi}}\\
0 & S_{\varphi}S_{\psi}\end{bmatrix}~\text{and}~~
\widehat{T}_{\varphi \psi}=\begin{bmatrix} T_{\varphi \psi} & {\Gamma^{*}_{\overline{\varphi \psi}}}\\
   0 &  S_{\varphi \psi}\end{bmatrix}.
\end{gather*}
Since
\begin{equation}\label{57}
\begin{split}
T_{\varphi}T_{\psi}u=&P\varphi P \psi u\\
=&P\varphi \psi u=T_{\varphi \psi}u
\end{split}
\end{equation}
and
\begin{equation}\label{33}
\begin{split}
(S_{\varphi}S_{\psi}\bar{v})(z)&=(Q-P)({\varphi}S_{\psi}\bar{v})(z)\\
&=\int_{\mathbb{T}^{2}}{\varphi}S_{\psi}\bar{v}(\mathbf{K}_{z}-1)d\sigma\\
&=\left< S_{\psi}\bar{v},\overline{{\varphi}(\mathbf{K}_{z}-1)}\right >\\
&=\left< (Q-P)(\psi\bar{v}),\overline{{\varphi}(\mathbf{K}_{z}-1)}\right >\\
&=\left< \psi\bar{v},\overline{{\varphi}(\mathbf{K}_{z}-1)}\right >\\
&=\int_{\mathbb{T}^{2}}{\varphi}\psi\bar{v}(\mathbf{K}_{z}-1)d\sigma\\
&=(Q-P)({\varphi}\psi\bar{v})(z)\\
&=(S_{{\varphi}\psi}\bar{v})(z),
\end{split}
\end{equation}
that is
\begin{equation}
\begin{split}
T_{\varphi}T_{\psi}=T_{\varphi \psi}~~\text{and}~~
S_{\varphi}S_{\psi}=S_{\varphi \psi}.
\end{split}
\end{equation}
Theeqrefore $\widehat{T}_{\varphi}\widehat{T}_{\psi}=\widehat{T}_{\varphi \psi}$ is equivalent to
$T_{\varphi}{\Gamma^{*}_{\bar{\psi}}}+{\Gamma^{*}_{\bar{\varphi}}}{S_{\psi}}=\Gamma^{*}_{\overline{\varphi \psi}}.$

$(c)\Leftrightarrow(d)$: Using the matrix representations of $\widehat{T}_{\bar{\psi}}\widehat{T}_{\bar{\varphi}}$
and $\widehat{T}_{\overline{\varphi \psi}}$ , we have
\begin{gather*}
\widehat{T}_{\bar{\psi}}\widehat{T}_{\bar{\varphi}}=\begin{bmatrix}T_{\bar{\psi}}T_{\bar{\varphi}} & 0\\
{\Gamma_{\bar{\psi}}}T_{\bar{\varphi}}+{S_{\bar{\psi}}}{\Gamma_{\bar{\varphi}}} & S_{\bar{\psi}}S_{\bar{\varphi}}\end{bmatrix}~\text{and}~~
\widehat{T}_{\overline{\varphi \psi}}=\begin{bmatrix} T_{\overline{\varphi \psi}} & 0\\
   \Gamma_{\overline{\varphi \psi}} &  S_{\overline{\varphi \psi}}\end{bmatrix}.
\end{gather*}

Observe that
\begin{equation}\label{35}
\begin{split}
S_{\bar{\psi}}S_{\bar{\varphi}}\bar{v}=&(Q-P)\bar{\psi}(Q-P)\overline{\varphi v}\\
=&(Q-P)\overline{\psi \varphi v}=S_{\overline{\varphi \psi}}\bar{v}
\end{split}
\end{equation}
and
\begin{equation}\label{34}
\begin{split}
(T_{\bar{\psi}}T_{\bar{\varphi}}u)(z)&=P(\bar{\psi}P\bar{\varphi}u)(z)\\
&=\int_{\mathbb{T}^{2}}\bar{\psi}P(\bar{\varphi}u)\overline{\mathbf{K}_{z}}d\sigma\\
&=\left< P(\bar{\varphi}u),\psi\mathbf{K}_{z}\right >\\
&=\left< \bar{\varphi}u,\psi\mathbf{K}_{z}\right >\\
&=\int_{\mathbb{T}^{2}} \bar{\varphi}u\overline{\psi\mathbf{K}_{z}}d\sigma\\
&=\int_{\mathbb{T}^{2}} \overline{\varphi \psi}u\overline{\mathbf{K}_{z}}d\sigma\\
&=(T_{\overline{\varphi \psi}}u)(z).
\end{split}
\end{equation}
Thus
\begin{equation}
\begin{split}
T_{\bar{\psi}}T_{\bar{\varphi}}=T_{\overline{\varphi \psi}}~~\text{and}~~
S_{\bar{\psi}}S_{\bar{\varphi}}=S_{\overline{\varphi \psi}}.
\end{split}
\end{equation}
Hence ~$\widehat{T}_{\bar{\psi}}\widehat{T}_{\bar{\varphi}}=\widehat{T}_{\overline{\varphi \psi}}$~
and~${\Gamma_{\bar{\psi}}}T_{\bar{\varphi}}+{S_{\bar{\psi}}}{\Gamma_{\bar{\varphi}}}=\Gamma_{\overline{\varphi \psi}}$
are equivalent.

$(a)\Leftrightarrow(c)$:
Since the span of $\{\mathbf{R}_{\eta}:\eta\in\mathbb{T}^{2}\}$ is dense in $h^{2}(\mathbb{T}^{2}),$  we have
\begin{equation}\label{32}
\begin{split}
(\widehat{T}_{\varphi}\widehat{T}_{\psi}\mathbf{R}_{\eta})(z)
&=Q(\varphi\widehat{T}_{\psi}\mathbf{R}_{\eta})(z)\\
&=\int_{\mathbb{T}^{2}}(\varphi\widehat{T}_{\psi}\mathbf{R}_{\eta})\mathbf{R}_{z}d\sigma\\
&=\left<\widehat{T}_{\psi}\mathbf{R}_{\eta},\bar{\varphi}\mathbf{R}_{z}\right>\\
&=\left<Q(\psi\mathbf{R}_{\eta}),Q(\bar{\varphi}\mathbf{R}_{z})\right>\\
&=\left<\psi\mathbf{R}_{\eta},\widehat{T}_{\bar{\varphi}}\mathbf{R}_{z}\right>\\
&=\int_{\mathbb{T}^{2}}\psi\mathbf{R}_{\eta}\overline{\widehat{T}_{\bar{\varphi}}\mathbf{R}_{z}}d\sigma\\
&=\overline{\int_{\mathbb{T}^{2}}\bar{\psi}(\widehat{T}_{\bar{\varphi}}\mathbf{R}_{z})\mathbf{R}_{\eta}d\sigma}\\
&=\overline{(\widehat{T}_{\bar{\psi}}\widehat{T}_{\bar{\varphi}}\mathbf{R}_{z})(\eta)}
\end{split}
\end{equation}
and
\begin{equation}
\begin{split}
(\widehat{T}_{\varphi \psi}\mathbf{R}_{\eta})(z)&=Q(\varphi \psi\mathbf{R}_{\eta})(z)\\
&=\int_{\mathbb{T}^{2}}\varphi \psi\mathbf{R}_{\eta}\mathbf{R}_{z}d\sigma\\
&=\overline{\int_{\mathbb{T}^{2}}\overline{\varphi \psi}\mathbf{R}_{z}\mathbf{R}_{\eta}d\sigma}\\
&=\overline{Q(\overline{\varphi \psi}\mathbf{R}_{z})(\eta)}\\
&=\overline{(\widehat{T}_{\overline{\varphi \psi}}\mathbf{R}_{z})(\eta)},
\end{split}
\end{equation}
it follows that $(\widehat{T}_{\varphi}\widehat{T}_{\psi}\mathbf{R}_{\eta})(z)=(\widehat{T}_{\varphi \psi}\mathbf{R}_{\eta})(z)$
if and only if
$\overline{(\widehat{T}_{\bar{\psi}}\widehat{T}_{\bar{\varphi}}\mathbf{R}_{z})(\eta)}=
\overline{(\widehat{T}_{\overline{\varphi \psi}}\mathbf{R}_{z})(\eta)}.$
So $\widehat{T}_{\varphi}\widehat{T}_{\psi}=\widehat{T}_{\varphi \psi}$ is equivalent to
$\widehat{T}_{\bar{\psi}}\widehat{T}_{\bar{\varphi}}=\widehat{T}_{\overline{\varphi \psi}}.$

$(a)\Leftrightarrow(e):$ see \cite[Theorem 2.1]{liu2013toeplitz}.
\end{proof}
Before stating  Lemma \ref{3}, we recall that for $h\in L^{2}(\mathbb{T}),$  the Hankel operator
$H_{h}$  is defined on $H^{2}(\mathbb{T})$ as follows:
\[H_{h}p=(I_{2}-P_{2})(hp),~~~~~~ p\in H^{2}(\mathbb{T}),\]
where $I_{2}$ is the identity operator on $ L^{2}(\mathbb{T}).$
We define $H^{*}_{h}$ by
\[H^{*}_{h}q=P_{2}(\bar{h}q),~~~~~~~~~~~~q\in {H^{2}(\mathbb{T})}^{\bot}.\]
Clearly $H_{h}$ and $H^{*}_{h}$ are densely defined.
\begin{lemma}\label{3}
Let $\varphi,\psi\in H^{2}(\mathbb{T}^{2})$.Then the following conditions are equivalent.

~~~~~~(a) $\widehat{T}_{\varphi}\widehat{T}_{\bar{\psi}}=\widehat{T}_{\bar{\psi}}\widehat{T}_{\varphi}.$

~~~~~~(b) $\widehat{T}_{\varphi}\widehat{T}_{\bar{\psi}}=\widehat{T}_{\varphi\bar{\psi}}.$

~~~~~~(c) $\widehat{T}_{\bar{\psi}}\widehat{T}_{\varphi}=\widehat{T}_{\varphi\bar{\psi}}.$

~~~~~~(d) $T_{\varphi}T_{\bar{\psi}}+{\Gamma^{*}_{\bar{\varphi}}}{\Gamma_{\bar{\psi}}}=T_{\varphi\bar{\psi}}.$

~~~~~~(e) $S_{\bar{\psi}}S_{\varphi}+{\Gamma_{\bar{\psi}}}{\Gamma^{*}_{\bar{\varphi}}}=S_{\varphi\bar{\psi}}.$

~~~~~~(f)~(f1) $\varphi$ is a function of one variable $z_{1}$ and $\psi$ is a function of one variable $z_{2};$ or

~~~~~~~~~~(f2) $\varphi$ is a function of one variable $z_{2}$ and $\psi$ is a function of one variable $z_{1};$ or

~~~~~~~~~~(f2) $\varphi$ or $\psi$ is constant.

\end{lemma}
\begin{proof}
$(a)\Leftrightarrow(b)\Leftrightarrow(c)$:
The proof is very similar to the proof
 \cite[Theorem 3.1]{ding2008question} and we omit it.

$(b)\Leftrightarrow(d)$: Using the matrix representations of $\widehat{T}_{\varphi}\widehat{T}_{\bar{\psi}}$
and $\widehat{T}_{\varphi\bar{\psi}},$  we have
\begin{gather*}
\widehat{T}_{\varphi}\widehat{T}_{\bar{\psi}}=\begin{bmatrix}T_{\varphi}T_{\bar{\psi}}+{\Gamma^{*}_{\bar{\varphi}}}{\Gamma_{\bar{\psi}}} & {\Gamma^{*}_{\bar{\varphi}}}{S_{\bar{\psi}}}\\
S_{\varphi}{\Gamma_{\bar{\psi}}} & S_{\varphi}S_{\bar{\psi}}\end{bmatrix}~\text{and}~~
\widehat{T}_{\varphi\bar{\psi}}=\begin{bmatrix}T_{\varphi\bar{\psi}} & {\Gamma^{*}_{\bar{\varphi}\psi}}\\
\Gamma_{\varphi\bar{\psi}} & S_{\varphi\bar{\psi}}\end{bmatrix}.
\end{gather*}
By Lemma \ref{4},
\begin{equation*}
\begin{split}
S_{\varphi}{\Gamma_{\bar{\psi}}}=\Gamma_{\varphi\bar{\psi}}~~~\text{and}~~~
{\Gamma^{*}_{\bar{\varphi}}}{S_{\bar{\psi}}}= {\Gamma^{*}_{\bar{\varphi}\psi}}.
\end{split}
\end{equation*}
Since
\begin{equation*}
\begin{split}
S_{\varphi}S_{\bar{\psi}}\bar{v}=&(Q-P)\varphi(Q-P)(\bar{\psi}\bar{v})\\
=&(Q-P)(\varphi\overline{\psi v})=S_{\varphi\bar{\psi}}\bar{v},
\end{split}
\end{equation*}
that is
\begin{equation}
\begin{split}
S_{\varphi}S_{\bar{\psi}}=S_{\varphi\bar{\psi}}.
\end{split}
\end{equation}
Thus ~$\widehat{T}_{\varphi}\widehat{T}_{\bar{\psi}}=\widehat{T}_{\varphi\bar{\psi}}$~
and ~$T_{\varphi}T_{\bar{\psi}}+{\Gamma^{*}_{\bar{\varphi}}}{\Gamma_{\bar{\psi}}}=T_{\varphi\bar{\psi}}$~
are equivalent.

$(c)\Leftrightarrow(e)$: By the matrix representations of $\widehat{T}_{\bar{\psi}}\widehat{T}_{\varphi}$
and $\widehat{T}_{\varphi\bar{\psi}}$ , we have
\begin{gather*}
\widehat{T}_{\bar{\psi}}\widehat{T}_{\varphi}=\begin{bmatrix}T_{\bar{\psi}}T_{\varphi} & {T_{\bar{\psi}}}{\Gamma^{*}_{\bar{\varphi}}}\\
{\Gamma_{\bar{\psi}}}T_{\varphi} & S_{\bar{\psi}}S_{\varphi}+{\Gamma_{\bar{\psi}}}{\Gamma^{*}_{\bar{\varphi}}}\end{bmatrix}~\text{and}~~
\widehat{T}_{\varphi\bar{\psi}}=\begin{bmatrix}T_{\varphi\bar{\psi}} & {\Gamma^{*}_{\bar{\varphi}\psi}}\\
\Gamma_{\varphi\bar{\psi}} & S_{\varphi\bar{\psi}}\end{bmatrix}.
\end{gather*}

Moreover, Lemma \ref{4} implies that
\begin{equation*}
\begin{split}
{T_{\bar{\psi}}}{\Gamma^{*}_{\bar{\varphi}}}={\Gamma^{*}_{\bar{\varphi}\psi}}~\text{and}~
{\Gamma_{\bar{\psi}}}T_{\varphi}= \Gamma_{\varphi\bar{\psi}}.
\end{split}
\end{equation*}
Since
\begin{equation*}
\begin{split}
T_{\bar{\psi}}T_{\varphi}u=&P\bar{\psi}P\varphi u\\
=&P(\bar{\psi}\varphi u)=T_{\varphi\bar{\psi}}u,
\end{split}
\end{equation*}
that is
\begin{equation}
\begin{split}
T_{\bar{\psi}}T_{\varphi}=T_{\varphi\bar{\psi}}.
\end{split}
\end{equation}
Hence ~$\widehat{T}_{\bar{\psi}}\widehat{T}_{\varphi}=\widehat{T}_{\varphi\bar{\psi}}$~
and ~$S_{\bar{\psi}}S_{\varphi}+{\Gamma_{\bar{\psi}}}{\Gamma^{*}_{\bar{\varphi}}}=S_{\varphi\bar{\psi}}$~
are equivalent.

$(d)\Rightarrow(f)$: Authors prove this result for bounded symbols in \cite[Theorem 1]{Sang2016toeplitz}.
For the sake of  reader's convenience, we will prove it again without the condition of bounded symobls.

Write
\begin{equation*}
\varphi= \sum_{i=0}^{+\infty}\varphi_{i}(z_{2})z^{i}_{1},~~
\psi= \sum_{j=0}^{+\infty}\psi_{j}(z_{2})z^{j}_{1}.
\end{equation*}

Let $\alpha,\beta,k,l \in\mathbb{Z_{+}}$ and without loss of generality we assume $k\geq l,$
\[L_{k,l}\triangleq  \left< (T_{\varphi\bar{\psi}}-T_{\varphi}T_{\bar{\psi}}){z^{k}_{1}}{z^{\alpha}_{2}},{z^{l}_{1}}{z^{\beta}_{2}}\right>\]
and \[R_{k,l}\triangleq  \left<\Gamma_{\bar{\varphi}}^{*}\Gamma_{\bar{\psi}}{z^{k}_{1}}{z^{\alpha}_{2}} ,{z^{l}_{1}}{z^{\beta}_{2}}\right>.\]
Using the similar method discussed in \cite{gu1997semi}, we have
\begin{equation}\label{50}
\begin{split}
L_{k,l}=&\left< (T_{\varphi\bar{\psi}}-T_{\varphi}T_{\bar{\psi}})z_{1}^{k}z^{\alpha}_{2} ,z_{1}^{l}z^{\beta}_{2}\right>\\
=&\left<\bar{\psi}z_{1}^{k}z^{\alpha}_{2},\bar{\varphi}z_{1}^{l}z^{\beta}_{2}\right>
-\left<P(\bar{\psi}z_{1}^{k}z^{\alpha}_{2}),\bar{\varphi}z_{1}^{l}z^{\beta}_{2}\right>\\
=&\left< \sum_{j=0}^{+\infty}\overline{\psi_{j}}({z_{2}})z^{\alpha}_{2}\bar{z}^{j}_{1}z_{1}^{k},
\sum_{i=0}^{+\infty}\overline{\varphi_{i}}({z_{2}})z^{\beta}_{2}\bar{z_{1}}^{i}z_{1}^{l}\right >\\
&-\left< \sum_{j=0}^{k}P_{2}(\overline{\psi_{j}}({z_{2}})z^{\alpha}_{2})z_{1}^{k-j},
\sum_{i=0}^{l}P_{2}(\overline{\varphi_{i}}({z_{2}}){z_{2}}^{\beta})z_{1}^{l-i} \right >\\
=&\sum_{j=0}^{+\infty}\left<\overline{\psi_{j}}({z_{2}})z^{\alpha}_{2},\overline{\varphi}_{l-k+j}({z_{2}})z^{\beta}_{2} \right>\\
&-\sum_{j=0}^{k}\left< P_{2}(\overline{\psi_{j}}({z_{2}})z^{\alpha}_{2}),\overline{\varphi}_{l-k+j}({z_{2}})z^{\beta}_{2}\right >
\end{split}
\end{equation}
and
\begin{equation}\label{51}
\begin{split}
R_{k,l}=&\left< \Gamma_{\bar{\psi}}z^{k}_{1}z^{\alpha}_{2} , \Gamma_{\bar{\varphi}}{z^{l}_{1}}z^{\beta}_{2}\right>\\
=&\left<P^{-}(\bar{\psi}z_{1}^{k}z^{\alpha}_{2})-P(\bar{\psi}z_{1}^{k}z^{\alpha}_{2})(0), \bar{\varphi}z^{l}_{1}z^{\beta}_{2}\right>\\
=&\left< \sum_{j=k}^{+\infty}P^{-}_{2}(\overline{\psi_{j}}({z_{2}})z^{\alpha}_{2})\bar{z}^{j-k}_{1},
 \sum_{i=l}^{+\infty}{P^{-}_{2}}(\overline{\varphi_{i}}({z_{2}})z^{\beta}_{2})\bar{z_{1}}^{i-l}\right>\\
&-\left<  P_{2}(\overline{\psi_{k}}({z_{2}})z^{\alpha}_{2})(0), \overline{\varphi_{l}}({z_{2}})z^{\beta}_{2}\right>\\
=&\sum_{j=k}^{+\infty}\left< {P^{-}_{2}}(\overline{\psi_{j}}({z_{2}})z^{\alpha}_{2}),\overline{\varphi}_{l-k+j}({z_{2}})z^{\beta}_{2}\right>\\
&-\left< P_{2}(\overline{\psi_{k}}({z_{2}})z^{\alpha}_{2})(0), \overline{\varphi_{l}}({z_{2}})z^{\beta}_{2}\right>.
\end{split}
\end{equation}
Furthermore, \eqref{50} and \eqref{51} imply that
\begin{equation*}
\begin{split}
L_{k,l}-L_{k+1,l+1}=R_{k,l}-R_{k+1,l+1},
\end{split}
\end{equation*}
that is,
\begin{equation*}
\begin{split}
&\left<  P^{-}_{2}(\overline{\psi}_{k}({z_{2}})z^{\alpha}_{2})-P_{2}(\overline{\psi}_{k}({z_{2}})z^{\alpha}_{2})(0)
,\overline{\varphi}_{l}({z_{2}}){z}^{\beta}_{2} \right>\\
=&\left< P_{2}(\overline{\psi}_{k+1}({z_{2}})z^{\alpha}_{2})-P_{2}(\overline{\psi}_{k+1}({z_{2}})z^{\alpha}_{2})(0)
,\overline{\varphi}_{l+1}({z_{2}})z^{\beta}_{2} \right>,\\
\end{split}
\end{equation*}
it follows that
\[\left< H_{\overline{\psi}_{k}}{z_{2}}^{\alpha},H_{\overline{\varphi}_{l}}{z_{2}}^{\beta}\right>
=\left< (T_{\varphi_{l+1}}T_{\overline{\psi}_{k+1}}-{\varphi_{l+1}}\otimes{\psi_{k+1}}){z_{2}}^{\alpha},{z_{2}}^{\beta}\right>
.\]\
Thus
\begin{equation*}
\begin{split}
H_{\overline{\varphi}_{l}}^{*}H_{\overline{\psi}_{k}}=T_{\varphi_{l+1}}T_{\overline{\psi}_{k+1}}-{\varphi_{l+1}}\otimes{\psi_{k+1}}
\end{split}
\end{equation*}
holds on $H^{2}(\mathbb{T}).$
We calculate the Berezin transform of $H_{\overline{\varphi_{l}}}^{*}H_{\overline{\psi_{k}}}$,
$T_{\varphi_{l+1}}T_{\overline{\psi}_{k+1}}$
and ${\varphi_{l+1}}\otimes{\psi_{k+1}}$ to get
\begin{equation*}
\begin{split}
\left<  H_{\overline{\varphi_{l}}}^{*}H_{\overline{\psi_{k}}}k_{\lambda_{2}} ,k_{\lambda_{2}}\right>
&=\left< H_{\overline{\psi_{k}}}k_{\lambda_{2}} ,H_{\overline{\varphi_{l}}}k_{\lambda_{2}}\right>\\
&=\left<  (I_{2}-P_{2})\overline{\psi_{k}}k_{\lambda_{2}} ,\overline{\varphi_{l}}k_{\lambda_{2}}\right> \\
&=\mathcal{P}[\varphi_{l}\overline{\psi_{k}}](\lambda_{2})-\varphi_{l}(\lambda_{2}) \overline{\psi}_{k}(\lambda_{2}),
\end{split}
\end{equation*}
\begin{equation*}
\begin{split}
\left< ({\varphi_{l+1}}\otimes{\psi_{k+1}}) k_{\lambda_{2}} ,k_{\lambda_{2}}\right>
&=\left< \left< k_{\lambda_{2}},{\psi_{k+1}} \right>{\varphi_{l+1}} ,k_{\lambda_{2}}\right>\\
&=\left< k_{\lambda_{2}},{\psi_{k+1}} \right>\left< {\varphi_{l+1}} ,k_{\lambda_{2}}\right> \\
&=P_{2}(\overline{\psi}_{k+1}k_{\lambda_{2}})(0)\overline{P_{2}(\overline{\varphi}_{l+1}k_{\lambda_{2}})(0)}\\
&=(1-|\lambda_{2}|^{2})\varphi_{l+1}(\lambda_{2})\overline{\psi}_{k+1}(\lambda_{2}),
\end{split}
\end{equation*}
and
\begin{equation*}
\begin{split}
\left<  T_{\varphi_{l+1}}T_{\overline{\psi}_{k+1}}k_{\lambda_{2}},k_{\lambda_{2}}\right>
&=\left< T_{\overline{\psi}_{k+1}}k_{\lambda_{2}} ,T_{\overline{\varphi}_{l+1}}k_{\lambda_{2}}\right>\\
&=\left< P_{2}(\overline{\psi}_{k+1}k_{\lambda_{2}}) ,P_{2}(\overline{\varphi}_{l+1}k_{\lambda_{2}})\right> \\
&=\varphi_{l+1}(\lambda_{2})\overline{\psi}_{k+1}(\lambda_{2}).
\end{split}
\end{equation*}

Since
\begin{equation*}
\begin{split}
\lim_{|\lambda_{2}|\rightarrow{1^{-}}}\bigg(\mathcal{P}[{\varphi}_{l}\overline{\psi_{k}}](\lambda_{2})-{\varphi}_{l}(\lambda_{2}) \overline{\psi}_{k}(\lambda_{2})\bigg)=0
\end{split}
\end{equation*}
and
\begin{equation*}
\begin{split}
\lim_{|\lambda_{2}|\rightarrow{1^{-}}}(1-|\lambda_{2}|^{2})\varphi_{l+1}(\lambda_{2})\overline{\psi}_{k+1}(\lambda_{2})=0,
\end{split}
\end{equation*}
we have
\[\lim_{|\lambda_{2}|\rightarrow{1^{-}}}\varphi_{l+1}(\lambda_{2})\overline{\psi}_{k+1}(\lambda_{2})=0,\]
it follows that\[\varphi_{l+1}\overline{\psi}_{k+1}=0 ~\text{a.e. on}~\mathbb{T}.\]
Theeqrefore either $\varphi_{l+1}(z_{2})=0$ for all $l\geq{0} $ or $\psi_{k+1}(z_{2})=0 $ for all $k\geq{0}$.
In this case, either $\varphi$ or $\psi$ is constant with respect to variable $z_{1}.$
Using the same argument as before, so either $\varphi$ or $\psi$ is constant with respect to variable $z_{2}.$

$(f)\Rightarrow(b)$: See \cite[Theorem 4.2]{liu2013toeplitz}.
\end{proof}
\begin{lemma}\label{8}
Let $\varphi,\psi\in H^{2}(\mathbb{T}^{2})$.Then the following conditions are equivalent.

~~~~~~~~~~~~(a) $\widehat{T}_{\varphi}\widehat{T}_{\psi}=\widehat{T}_{\psi}\widehat{T}_{\varphi}.$

~~~~~~~~~~~~(b) $T_{\varphi}{\Gamma^{*}_{\bar{\psi}}}+{\Gamma^{*}_{\bar{\varphi}}}{S_{\psi}}=T_{\psi}{\Gamma^{*}_{\bar{\varphi}}}+{\Gamma^{*}_{\bar{\psi}}}{S_{\varphi}}.$

~~~~~~~~~~~~(c) $\widehat{T}_{\bar{\varphi}}\widehat{T}_{\bar{\psi}}=\widehat{T}_{\bar{\psi}}\widehat{T}_{\bar{\varphi}}.$

~~~~~~~~~~~~(d) ${\Gamma_{\bar{\varphi}}}T_{\bar{\psi}}+{S_{\bar{\varphi}}}{\Gamma_{\bar{\psi}}}={\Gamma_{\bar{\psi}}}T_{\bar{\varphi}}+{S_{\bar{\psi}}}{\Gamma_{\bar{\varphi}}}.$

~~~~~~~~~~~~(e)~(e1) Both $\varphi$ and $\psi$ are functions of one variable $z_{1}$; or

~~~~~~~~~~~~~~~~~(e2)  Both $\varphi$ and $\psi$ are functions of one variable $z_{2}$; or

~~~~~~~~~~~~~~~~~(e2) A nontrivial linear combination of $\varphi$ and $\psi$ is constant.

\end{lemma}
\begin{proof}

$(a)\Leftrightarrow(b)$: By \eqref{37}, we have
\begin{gather*}
\widehat{T}_{\varphi}\widehat{T}_{\psi}=\begin{bmatrix}T_{\varphi}T_{\psi} & T_{\varphi}{\Gamma^{*}_{\bar{\psi}}}+{\Gamma^{*}_{\bar{\varphi}}}{S_{\psi}}\\
0 & S_{\varphi}S_{\psi}\end{bmatrix}~\text{and}~~
\widehat{T}_{\psi}\widehat{T}_{\varphi}=\begin{bmatrix}T_{\psi}T_{\varphi} & T_{\psi}{\Gamma^{*}_{\bar{\varphi}}}+{\Gamma^{*}_{\bar{\psi}}}{S_{\varphi}}\\
0 & S_{\psi}S_{\varphi}\end{bmatrix}.\quad
\end{gather*}
Since
\begin{equation*}
\begin{split}
T_{\varphi}T_{\psi}u=&P{\varphi}P{\psi}u\\
=&P{\varphi}{\psi}u\\
=&P{\psi}P{\varphi}u=T_{\psi}T_{\varphi}u,
\end{split}
\end{equation*}
it follows that $T_{\varphi}T_{\psi}=T_{\psi}T_{\varphi}.$
By \eqref{33}, we have
$S_{\varphi}S_{\psi}=S_{\psi \varphi}$ and $S_{\psi}S_{\varphi}=S_{\psi \varphi},$ so that $S_{\varphi}S_{\psi}=S_{\psi}S_{\varphi}.$
Thus $\widehat{T}_{\varphi}\widehat{T}_{\psi}=\widehat{T}_{\psi}\widehat{T}_{\varphi}$ is equivalent to
$T_{\varphi}{\Gamma^{*}_{\bar{\psi}}}+{\Gamma^{*}_{\bar{\varphi}}}{S_{\psi}}=T_{\psi}{\Gamma^{*}_{\bar{\varphi}}}+{\Gamma^{*}_{\bar{\psi}}}{S_{\varphi}}.$

$(c)\Leftrightarrow(d)$: Again by \eqref{37}, we have
\begin{gather*}
\widehat{T}_{\bar{\varphi}}\widehat{T}_{\bar{\psi}}=\begin{bmatrix}T_{\bar{\varphi}}T_{\bar{\psi}} & 0\\
{\Gamma_{\bar{\varphi}}}T_{\bar{\psi}}+{S_{\bar{\varphi}}}{\Gamma_{\bar{\psi}}} & S_{\bar{\varphi}}S_{\bar{\psi}}\end{bmatrix}~\text{and}~~
\widehat{T}_{\bar{\psi}}\widehat{T}_{\bar{\varphi}}=\begin{bmatrix}T_{\bar{\psi}}T_{\bar{\varphi}} & 0\\
{\Gamma_{\bar{\psi}}}T_{\bar{\varphi}}+{S_{\bar{\psi}}}{\Gamma_{\bar{\varphi}}} & S_{\bar{\psi}}S_{\bar{\varphi}}\end{bmatrix}.\quad
\end{gather*}

By \eqref{34},
$T_{\bar{\varphi}}T_{\bar{\psi}}=T_{\overline{\psi \varphi}}$  and  $T_{\bar{\psi}}T_{\bar{\varphi}}=T_{\overline{\psi \varphi}},$
that is $T_{\bar{\varphi}}T_{\bar{\psi}}=T_{\bar{\psi}}T_{\bar{\varphi}}.$
Using \eqref{35}, we have
$S_{\bar{\varphi}}S_{\bar{\psi}}=S_{\overline{\psi \varphi}}$
and $S_{\bar{\psi}}S_{\bar{\varphi}}=S_{\overline{\psi \varphi}},$
so $S_{\bar{\varphi}}S_{\bar{\psi}}=S_{\bar{\psi}}S_{\bar{\varphi}}.$
It follows that $\widehat{T}_{\bar{\varphi}}\widehat{T}_{\bar{\psi}}=\widehat{T}_{\bar{\psi}}\widehat{T}_{\bar{\varphi}}$ and
${\Gamma_{\bar{\varphi}}}T_{\bar{\psi}}+{S_{\bar{\varphi}}}{\Gamma_{\bar{\psi}}}={\Gamma_{\bar{\psi}}}T_{\bar{\varphi}}+{S_{\bar{\psi}}}{\Gamma_{\bar{\varphi}}}$
are equivalent.

$(a)\Leftrightarrow(c)$: By \eqref{32}, we have
$(\widehat{T}_{\varphi}\widehat{T}_{\psi}\mathbf{R}_{\eta})(z)=\overline{(\widehat{T}_{\bar{\psi}}\widehat{T}_{\bar{\varphi}}\mathbf{R}_{z})(\eta)},$
and similarly
$(\widehat{T}_{\psi}\widehat{T}_{\varphi}\mathbf{R}_{\eta})(z)=\overline{(\widehat{T}_{\bar{\varphi}}\widehat{T}_{\bar{\psi}}\mathbf{R}_{z})(\eta)},$
so $\widehat{T}_{\varphi}\widehat{T}_{\psi}=\widehat{T}_{\psi}\widehat{T}_{\varphi}$ and
$\widehat{T}_{\bar{\varphi}}\widehat{T}_{\bar{\psi}}=\widehat{T}_{\bar{\psi}}\widehat{T}_{\bar{\varphi}}$ are equivalent.

$(a)\Leftrightarrow(e)$: See \cite[Theorem 3.1]{liu2013toeplitz}.
\end{proof}
\section{Semi-Commuting Toeplitz operators}
\hspace*{\parindent}

{\bf{Theorem \ref{27}}} can be restated as follows:

Let $f,g$ be two bounded functions in $ h^{2}(\mathbb{T}^{2})$, and $f=f_{+}+f_{-},~g=g_{+}+g_{-},$ where $f_{+},~\bar{f}_{-},~g_{+},~\bar{g}_{-}\in H^{2}(\mathbb{T}^{2}).$
Then ~$\widehat{T}_f\widehat{T}_g=\widehat{T}_{fg}$~ if and only if one of the following conditions is satisfied:

(A)  $f_{+}=f_{+}(z_{1}),~f_{-}=f_{-}(z_{2}),~g_{+}=g_{+}(z_{1}),~g_{-}=g_{-}(z_{2});$

(B)  $f_{+}=f_{+}(z_{2}),~f_{-}=f_{-}(z_{1}),~g_{+}=g_{+}(z_{2}),~g_{-}=g_{-}(z_{1});$

(C) Either $f$ or $g$ is constant.

\begin{proof}
Suppose that $\widehat{T}_f\widehat{T}_g=\widehat{T}_{fg},$ by the matrix representations of  $\widehat{T}_{f}\widehat{T}_{g}$  \eqref{37}
and  $\widehat{T}_{fg}$ \eqref{36}, we have
\begin{gather*}
\widehat{T}_{f}\widehat{T}_{g}=\begin{bmatrix}T_{f}T_{g}+{\Gamma^{*}_{\bar{f}}}{\Gamma_{g}} & T_{f}{\Gamma^{*}_{\bar{g}}}+{\Gamma^{*}_{\bar{f}}}{S_{g}}\\
\Gamma_{f}T_{g}+S_{f}{\Gamma_{g}} & \Gamma_{f}{\Gamma^{*}_{\bar{g}}}+S_{f}S_{g}\end{bmatrix}
=\begin{bmatrix} T_{fg} & {\Gamma^{*}_{\overline{fg}}}\\
   \Gamma_{fg} &  S_{fg}\end{bmatrix}=\widehat{T}_{fg},
\end{gather*}
hence
\begin{equation}\label{6}
\begin{split}
T_{f}T_{g}+{\Gamma^{*}_{\bar{f}}}{\Gamma_{g}}&=T_{fg},\\
T_{f}{\Gamma^{*}_{\bar{g}}}+{\Gamma^{*}_{\bar{f}}}{S_{g}}&= {\Gamma^{*}_{\overline{fg}}},\\
\Gamma_{f}T_{g}+S_{f}{\Gamma_{g}}&=\Gamma_{fg},\\
\Gamma_{f}{\Gamma^{*}_{\bar{g}}}+S_{f}S_{g}&=S_{fg}.
\end{split}
\end{equation}

By substituting $f=f_{+}+f_{-}$ and $ g=g_{+}+g_{-}$ into \eqref{6}, we obtain
\begin{equation}\label{12}
\begin{split}
&T_{f_{+}}T_{g_{+}}+T_{f_{+}}T_{g_{-}}+T_{f_{-}}T_{g_{+}}+T_{f_{-}}T_{g_{-}}
+{\Gamma^{*}_{\overline{f_{+}}}}{\Gamma_{g_{-}}}\\
=&T_{f_{+}g_{+}}+T_{f_{+}g_{-}}+T_{f_{-}g_{+}}+T_{f_{-}g_{-}},\\
&T_{f_{+}}{\Gamma^{*}_{\bar{g}_{+}}}+T_{f_{-}}{\Gamma^{*}_{\bar{g}_{+}}}
+{\Gamma^{*}_{\bar{f}_{+}}}{S_{g_{+}}}+{\Gamma^{*}_{\bar{f}_{+}}}{S_{g_{-}}}\\
=&{\Gamma^{*}_{\overline{f_{+}g_{+}}}}+\Gamma^{*}_{\overline{f_{+}g_{-}}}+\Gamma^{*}_{\overline{{f_{-}g_{+}}}},\\
&\Gamma_{f_{-}}T_{g_{+}}+\Gamma_{f_{-}}T_{g_{-}}
+S_{f_{+}}{\Gamma_{g_{-}}}+S_{f_{-}}{\Gamma_{g_{-}}}\\
=&\Gamma_{f_{+}g_{-}}+\Gamma_{{f_{-}}g_{+}}+\Gamma_{{f_{-}g_{-}}},\\
&\Gamma_{f_{-}}{\Gamma^{*}_{\overline{g_{+}}}}+S_{f_{+}}S_{g_{+}}+S_{f_{+}}S_{g_{-}}+S_{f_{-}}S_{g_{+}}+S_{f_{-}}S_{g_{-}}\\
=&S_{f_{+}g_{+}}+S_{f_{+}g_{-}}+S_{f_{-}g_{+}}+S_{f_{-}g_{-}}.\\
\end{split}
\end{equation}

By \eqref{57}, \eqref{34}, \eqref{33} and \eqref{35} we have
\begin{equation}
\begin{split}
T_{f_{+}}T_{g_{+}}&=T_{f_{+}g_{+}},\\
T_{f_{-}}T_{g_{+}}&=T_{f_{-}g_{+}},\\
T_{f_{-}}T_{g_{-}}&=T_{f_{-}g_{-}},\\
S_{f_{+}}S_{g_{+}}&=S_{f_{+}g_{+}},\\
S_{f_{+}}S_{g_{-}}&=S_{f_{+}g_{-}},\\
S_{f_{-}}S_{g_{-}}&=S_{f_{-}g_{-}}.
\end{split}
\end{equation}

Using Lemma \ref{4}, we have
\begin{equation}
\begin{split}
T_{f_{-}}{\Gamma^{*}_{\bar{g}_{+}}}&=\Gamma^{*}_{\overline{f_{-}g_{+}}},\\
{\Gamma^{*}_{\bar{f}_{+}}}{S_{{g}_{-}}}&={\Gamma^{*}_{\overline{{f}_{+}g_{-}}}},\\
\Gamma_{f_{-}}T_{g_{+}}&=\Gamma_{f_{-}g_{+}},\\
S_{f_{+}}{\Gamma_{g_{-}}}&=\Gamma_{f_{+}g_{-}}.
\end{split}
\end{equation}
So that \eqref{12} becomes
\begin{equation}\label{13}
\begin{split}
T_{f_{+}}T_{g_{-}}+{\Gamma^{*}_{\overline{f_{+}}}}{\Gamma_{g_{-}}}
&=T_{f_{+}g_{-}},
\end{split}
\end{equation}
\begin{equation}\label{56}
\begin{split}
T_{f_{+}}{\Gamma^{*}_{\bar{g}_{+}}}
+{\Gamma^{*}_{\bar{f}_{+}}}{S_{g_{+}}}&=
{\Gamma^{*}_{\overline{f_{+}g_{+}}}},\\
\end{split}
\end{equation}
\begin{equation}\label{58}
\begin{split}
\Gamma_{f_{-}}T_{g_{-}}
+S_{f_{-}}{\Gamma_{g_{-}}}
&=\Gamma_{f_{-}g_{-}},\\
\end{split}
\end{equation}
\begin{equation}\label{59}
\begin{split}
S_{f_{-}}S_{g_{+}}+\Gamma_{f_{-}}{\Gamma^{*}_{\overline{g_{+}}}}
&=S_{f_{-}g_{+}}.\\
\end{split}
\end{equation}
Applying Lemma \ref{2} to \eqref{56} and \eqref{58}, we obtain
\begin{equation}\label{38}
\begin{split}
\widehat{T}_{f_{+}}\widehat{T}_{g_{+}}&=\widehat{T}_{f_{+}g_{+}},\\
\widehat{T}_{f_{-}}\widehat{T}_{g_{-}}&=\widehat{T}_{f_{-}g_{-}}.
\end{split}
\end{equation}

An application of  Lemma \ref{2} to \eqref{38}, we see that $f_{+},~f_{-},~g_{+}~\text{and}~g_{-}$ satisfy
\begin{equation}\label{14}
\begin{split}
\rm \left\{ \begin{array}{lll}
   \ \textcircled{\small{a}}~ f_{+}=f_{+}(z_{1}),~g_{+}=g_{+}(z_{1});\text{or}  \\
     \ \textcircled{\small{b}} ~f_{+}=f_{+}(z_{2}),~g_{+}=g_{+}(z_{2});\text{or}\\
      \ \textcircled{\small{c}} ~f_{+}~\text{or}~g_{+} ~\text{is constant},
     \end{array} \right.
\end{split}
\end{equation}
and
\begin{equation}\label{41}
\begin{split}
\rm \left\{ \begin{array}{lll}
     \ \textcircled{\small{d}}~ f_{-}=f_{-}(z_{1}),~g_{-}=g_{-}(z_{1});\text{or}  \\
      \ \textcircled{\small{e}} ~f_{-}=f_{-}(z_{2}),~g_{-}=g_{-}(z_{2});\text{or} \\
       \ \textcircled{\small{f}} ~f_{-}~\text{or}~g_{-} ~\text{is constant}.
     \end{array} \right.
\end{split}
\end{equation}

Applying Lemma \ref{3} to \eqref{13} and \eqref{59}, we obtain
\begin{equation}\label{39}
\begin{split}
\widehat{T}_{f_{+}}\widehat{T}_{g_{-}}&=\widehat{T}_{f_{+}g_{-}},\\
\widehat{T}_{f_{-}}\widehat{T}_{g_{+}}&=\widehat{T}_{f_{-}g_{+}}.
\end{split}
\end{equation}

An application of Lemma \ref{3} to \eqref{39}, we see that $f_{+},~f_{-},~g_{+}~\text{and}~g_{-}$ satisfy
\begin{equation}\label{42}
\begin{split}
\rm \left\{ \begin{array}{lll}
   \  \textcircled{\small{1}} f_{+}=f_{+}(z_{1}),~g_{-}=g_{-}(z_{2});\text{or}  \\
     \ \textcircled{\small{2}} f_{+}=f_{+}(z_{2}),~g_{-}=g_{-}(z_{1});\text{or}\\
      \ \textcircled{\small{3}} f_{+}~\text{or}~g_{-} ~\text{is constant},
     \end{array} \right.
\end{split}
\end{equation}
and
\begin{equation}\label{43}
\begin{split}
\rm \left\{ \begin{array}{lll}
     \ \textcircled{\small{4}}f_{-}=f_{-}(z_{1}),~g_{+}=g_{+}(z_{2});\text{or}  \\
      \ \textcircled{\small{5}}f_{-}=f_{-}(z_{2}),~g_{+}=g_{+}(z_{1});\text{or} \\
       \ \textcircled{\small{6}}f_{-}~\text{or}~g_{+} ~\text{is constant}.
     \end{array} \right.
\end{split}
\end{equation}

Moreover, by substituting $f=f_{+}+f_{-}$ and $ g=g_{+}+g_{-}$ into $\widehat{T}_f\widehat{T}_g=\widehat{T}_{fg},$
we have
\begin{equation}\label{17}
\begin{split}
\widehat{T}_{f_{+}}\widehat{T}_{g_{+}}+\widehat{T}_{f_{+}}\widehat{T}_{g_{-}}
+\widehat{T}_{f_{-}}\widehat{T}_{g_{+}}+\widehat{T}_{f_{-}}\widehat{T}_{g_{-}}
=\widehat{T}_{f_{+}g_{+}}+\widehat{T}_{f_{+}g_{-}}
+\widehat{T}_{f_{-}g_{+}}+\widehat{T}_{f_{-}g_{-}}.
\end{split}
\end{equation}

Combining \eqref{38}, \eqref{39} and \eqref{17}, we conclude that $\widehat{T}_f\widehat{T}_g=\widehat{T}_{fg}$ if and only if
$\widehat{T}_{f_{+}}\widehat{T}_{g_{+}}=\widehat{T}_{f_{+}g_{+}},
~\widehat{T}_{f_{-}}\widehat{T}_{g_{-}}=\widehat{T}_{f_{-}g_{-}},
~\widehat{T}_{f_{+}}\widehat{T}_{g_{-}}=\widehat{T}_{f_{+}g_{-}},
~\widehat{T}_{f_{-}}\widehat{T}_{g_{+}}=\widehat{T}_{f_{-}g_{+}}$
if and only if ~$f_{+},~f_{-},~g_{+}$, and $g_{-}$ simultaneously satisfy \eqref{14}, \eqref{41}, \eqref{42} and \eqref{43}.

According to \eqref{14} and \eqref{41}, $f_{+},~f_{-},~g_{+}$ and $g_{-}$ can be grouped into four cases.

\noindent{\bf{Case 1.1}}\,\,\,: \textcircled{\small{a}}\textcircled{\small{d}}, \textcircled{\small{b}}\textcircled{\small{e}};

\noindent{\bf{Case 1.2}}\,\,\,: \textcircled{\small{a}}\textcircled{\small{e}}, \textcircled{\small{b}}\textcircled{\small{d}};

\noindent{\bf{Case 1.3}}\,\,\,: \textcircled{\small{a}}\textcircled{\small{f}}, \textcircled{\small{b}}\textcircled{\small{f}},
           \textcircled{\small{c}}\textcircled{\small{d}}, \textcircled{\small{c}}\textcircled{\small{e}};

\noindent{\bf{Case 1.4}}\,\,\,: \textcircled{\small{c}}\textcircled{\small{f}}.

In addition, according to \eqref{42} and \eqref{43}, $f_{+},~f_{-},~g_{+},~g_{-}$ can be divided into other four cases.

\noindent{\bf{Case 2.1}}\,\,\,: \textcircled{\small{1}}\textcircled{\small{4}}, \textcircled{\small{2}}\textcircled{\small{5}};

\noindent{\bf{Case 2.2}}\,\,\,: \textcircled{\small{1}}\textcircled{\small{5}}, \textcircled{\small{2}}\textcircled{\small{4}};

\noindent{\bf{Case 2.3}}\,\,\,: \textcircled{\small{1}}\textcircled{\small{6}}, \textcircled{\small{2}}\textcircled{\small{6}},
          \textcircled{\small{3}}\textcircled{\small{4}}, \textcircled{\small{3}}\textcircled{\small{5}};

\noindent{\bf{Case 2.4}}\,\,\,: \textcircled{\small{3}}\textcircled{\small{6}}.

We denote case \textcircled{\small{a}}\textcircled{\small{d}} as the case where
$f_{+},~f_{-},~g_{+},~g_{-}$ both satisfy \textcircled{\small{a}} with \textcircled{\small{d}}.
We find that \textcircled{\small{a}}\textcircled{\small{e}} and \textcircled{\small{1}}\textcircled{\small{5}}
are the same, they are (A),
\textcircled{\small{b}}\textcircled{\small{d}} and \textcircled{\small{2}}\textcircled{\small{4}} are the same, they are (B),
combining \textcircled{\small{a}}\textcircled{\small{d}} with \textcircled{\small{2}}\textcircled{\small{4}} gives (C).
The following table presents all cases, as desired.
\begin{table}[h!]
\footnotesize
\begin{tabular}{|p{0.8cm}|cc|cc|cccc|c|}
\hline
Case
&\textcircled{\tiny{a}}\textcircled{\tiny{d}}
&\textcircled{\tiny{b}}\textcircled{\tiny{e}}
&\textcircled{\tiny{a}}\textcircled{\tiny{e}}
&\textcircled{\tiny{b}}\textcircled{\tiny{d}}
&\textcircled{\tiny{a}}\textcircled{\tiny{f}}
&\textcircled{\tiny{b}}\textcircled{\tiny{f}}
&\textcircled{\tiny{c}}\textcircled{\tiny{e}}
&\textcircled{\tiny{c}}\textcircled{\tiny{d}}
&\textcircled{\tiny{c}}\textcircled{\tiny{f}}\\ \hline
\textcircled{\tiny{1}}\textcircled{\tiny{4}}
&(C)
&(C)
&(A)
&(B)
&(A),(C)
&(B),(C)
&(A),(C)
&(B),(C)
&(A),(C) \\
\textcircled{\tiny{2}}\textcircled{\tiny{5}}
&(C)
&(C)
&(A)
&(B)
&(A),(C)
&(B),(C)
&(A),(C)
&(B),(C)
&(A),(C) \\ \hline
\textcircled{\tiny{1}}\textcircled{\tiny{5}}
&(A)
&(A)
&(A)
&(C)
&(A)
&(A)
&(A)
&(A)
&(A) \\
\textcircled{\tiny{2}}\textcircled{\tiny{4}}
&(B)
&(B)
&(C)
&(B)
&(B)
&(B)
&(B)
&(B)
&(B) \\ \hline
\textcircled{\tiny{1}}\textcircled{\tiny{6}}
&(A),(C)
&(A),(C)
&(A)
&(B)
&(A),(C)
&(C)
&(A),(C)
&(C)
&(A),(C) \\
\textcircled{\tiny{3}}\textcircled{\tiny{5}}
&(A),(C)
&(A),(C)
&(A)
&(B)
&(A),(C)
&(C)
&(A),(C)
&(C)
&(A),(C) \\
\textcircled{\tiny{2}}\textcircled{\tiny{6}}
&(B),(C)
&(B),(C)
&(A)
&(B)
&(C)
&(B),(C)
&(C)
&(B),(C)
&(B),(C) \\
\textcircled{\tiny{3}}\textcircled{\tiny{4}}
&(B),(C)
&(B),(C)
&(A)
&(B)
&(C)
&(B),(C)
&(C)
&(B),(C)
&(B),(C) \\ \hline
\textcircled{\tiny{3}}\textcircled{\tiny{6}}
&(A),(B),(C)
&(A),(B),(C)
&(A)
&(B)
&(A),(C)
&(B),(C)
&(A),(C)
&(B),(C)
&(C) \\ \hline
\end{tabular}
~~~~~~~~~~~~~~
\caption{All Cases}
\label{table:1}
\end{table}
\end{proof}
\section{Commuting Toeplitz operators}
\hspace*{\parindent}

{\bf{Theorem \ref{26}}} can be eqreformulated in the following form.

Let $f,g\in h^{2}(\mathbb{T}^{2})$, and $f=f_{+}+f_{-},~g=g_{+}+g_{-},$ where $f_{+},\bar{f}_{-},g_{+},\bar{g}_{-}\in H^{2}(\mathbb{T}^{2}).$
Then $\widehat{T}_f\widehat{T}_g=\widehat{T}_g\widehat{T}_f$ if and only if one of the following conditions is satisfied:

~~(I)   ~$f_{+}=f_{+}(z_{1}),~f_{-}=f_{-}(z_{2}),~g_{+}=g_{+}(z_{1}),~g_{-}=g_{-}(z_{2});$

~(II)   ~$f_{+}=f_{+}(z_{2}),~f_{-}=f_{-}(z_{1}),~g_{+}=g_{+}(z_{2}),~g_{-}=g_{-}(z_{1});$

(III)  A nontrivial linear combination of $f$ and $g$ is constant.
\begin{proof}
If $\widehat{T}_{f}\widehat{T}_{g}=\widehat{T}_{g}\widehat{T}_{f},$
by  the matrix representations of $\widehat{T}_{f}\widehat{T}_{g}$ and $\widehat{T}_{g}\widehat{T}_{f}$
(see \eqref{37}), we have

\begin{equation*}
\widehat{T}_{f}\widehat{T}_{g}=\begin{bmatrix}T_{f}T_{g}+{\Gamma^{*}_{\bar{f}}}{\Gamma_{g}} & T_{f}{\Gamma^{*}_{\bar{g}}}+{\Gamma^{*}_{\bar{f}}}{S_{g}}\\
\Gamma_{f}T_{g}+S_{f}{\Gamma_{g}} & \Gamma_{f}{\Gamma^{*}_{\bar{g}}}+S_{f}S_{g}\end{bmatrix}
=\begin{bmatrix}T_{g}T_{f}+{\Gamma^{*}_{\bar{g}}}{\Gamma_{f}} & T_{g}{\Gamma^{*}_{\bar{f}}}+{\Gamma^{*}_{\bar{g}}}{S_{f}}\\
\Gamma_{g}T_{f}+S_{g}{\Gamma_{f}} & \Gamma_{g}{\Gamma^{*}_{\bar{f}}}+S_{g}S_{f}\end{bmatrix}=\widehat{T}_{g}\widehat{T}_{f}.
\end{equation*}
and hence
\begin{equation}\label{48}
\begin{split}
T_{f}{\Gamma^{*}_{\bar{g}}}+{\Gamma^{*}_{\bar{f}}}{S_{g}}
&= T_{g}{\Gamma^{*}_{\bar{f}}}+{\Gamma^{*}_{\bar{g}}}{S_{f}},\\
\Gamma_{f}T_{g}+S_{f}{\Gamma_{g}}&=\Gamma_{g}T_{f}+S_{g}{\Gamma_{f}}.
\end{split}
\end{equation}
By substituting $f=f_{+}+f_{-}$ and $g=g_{+}+g_{-}$ into \eqref{48}, we obtain
\begin{equation}\label{7}
\begin{split}
T_{f_{+}}{\Gamma^{*}_{\overline{g_{+}}}}+T_{f_{-}}{\Gamma^{*}_{\overline{g_{+}}}}
+{\Gamma^{*}_{\overline{f_{+}}}}{S_{g_{+}}}+{\Gamma^{*}_{\overline{f_{+}}}}{S_{g_{-}}}
&=T_{g_{+}}{\Gamma^{*}_{\overline{f_{+}}}}+T_{g_{-}}{\Gamma^{*}_{\overline{f_{+}}}}
+{\Gamma^{*}_{\overline{g_{+}}}}{S_{f_{+}}}+{\Gamma^{*}_{\overline{g_{+}}}}{S_{f_{-}}},\\
\Gamma_{f_{-}}T_{g_{+}}+\Gamma_{f_{-}}T_{g_{-}}+S_{f_{+}}{\Gamma_{g_{-}}}+S_{f_{-}}{\Gamma_{g_{-}}}
&=\Gamma_{g_{-}}T_{f_{+}}+\Gamma_{g_{-}}T_{f_{-}}+S_{g_{+}}{\Gamma_{f_{-}}}+S_{g_{-}}{\Gamma_{f_{-}}}.
\end{split}
\end{equation}
By Lemma \ref{4}, we have
\begin{equation*}
\begin{split}
T_{f_{-}}{\Gamma^{*}_{\overline{g_{+}}}}&={\Gamma^{*}_{\overline{g_{+}}}}{S_{f_{-}}},\\
{\Gamma^{*}_{\overline{f_{+}}}}{S_{g_{-}}}&=T_{g_{-}}{\Gamma^{*}_{\overline{f_{+}}}},\\
\Gamma_{f_{-}}T_{g_{+}}&=S_{g_{+}}{\Gamma_{f_{-}}},\\
S_{f_{+}}{\Gamma_{g_{-}}}&=\Gamma_{g_{-}}T_{f_{+}}.
\end{split}
\end{equation*}
Thus \eqref{7} becomes
\begin{equation}\label{62}
\begin{split}
T_{f_{+}}{\Gamma^{*}_{\overline{g_{+}}}}
+{\Gamma^{*}_{\overline{f_{+}}}}{S_{g_{+}}}
&=T_{g_{+}}{\Gamma^{*}_{\overline{f_{+}}}}
+{\Gamma^{*}_{\overline{g_{+}}}}{S_{f_{+}}},\\
\Gamma_{f_{-}}T_{g_{-}}+S_{f_{-}}{\Gamma_{g_{-}}}
&=\Gamma_{g_{-}}T_{f_{-}}+S_{g_{-}}{\Gamma_{f_{-}}}.
\end{split}
\end{equation}
Applying Lemma \ref{8} to \eqref{62}, we get
\begin{equation}
\begin{split}
\widehat{T}_{f_{+}}\widehat{T}_{g_{+}}&=\widehat{T}_{g_{+}}\widehat{T}_{f_{+}},\\
\widehat{T}_{f_{-}}\widehat{T}_{g_{-}}&=\widehat{T}_{g_{-}}\widehat{T}_{f_{-}}.
\end{split}
\end{equation}
Again by Lemma \ref{8} we see that $f_{+},~f_{-},~g_{+},~g_{-}$ satisfy
\begin{equation}\label{20}
\begin{split}
\rm \left\{ \begin{array}{lll}
 \ (a)~ f_{+}=f_{+}(z_{1}),g_{+}=g_{+}(z_{1});\text{or}  \\
 \ (b)~ f_{+}=f_{+}(z_{2}),g_{+}=g_{+}(z_{2});\text{or}\\
 \ (c)~ \text{a nontrivial linear combination of} ~f_{+}~ \text{and} ~g_{+}~ \text{is constant},\\
     \end{array} \right.
\end{split}
\end{equation}
and
\begin{equation}\label{49}
\begin{split}
\rm \left\{ \begin{array}{lll}
     \ (1)~f_{-}=f_{-}(z_{1}),g_{-}=g_{-}(z_{1});\text{or}  \\
      \ (2)~f_{-}=f_{-}(z_{2}),g_{-}=g_{-}(z_{2});\text{or} \\
       \ (3)~\text{a nontrivial linear combination of} ~f_{-}~ \text{and} ~g_{-}~ \text{is constant},.
     \end{array} \right.
\end{split}
\end{equation}

By assumption $\widehat{T}_f\widehat{T}_g=\widehat{T}_g\widehat{T}_f,$ we have
\begin{equation}\label{60}
\begin{split}
\widehat{T}_{f_{+}}\widehat{T}_{g_{+}}+\widehat{T}_{f_{+}}\widehat{T}_{g_{-}}
+\widehat{T}_{f_{-}}\widehat{T}_{g_{+}}+\widehat{T}_{f_{-}}\widehat{T}_{g_{-}}
=\widehat{T}_{g_{+}}\widehat{T}_{f_{+}}+\widehat{T}_{g_{+}}\widehat{T}_{f_{-}}
+\widehat{T}_{g_{-}}\widehat{T}_{f_{+}}+\widehat{T}_{g_{-}}\widehat{T}_{f_{-}}.
\end{split}
\end{equation}
Since $\widehat{T}_{f_{+}}\widehat{T}_{g_{+}}=\widehat{T}_{g_{+}}\widehat{T}_{f_{+}}$
and $\widehat{T}_{f_{-}}\widehat{T}_{g_{-}}=\widehat{T}_{g_{-}}\widehat{T}_{f_{-}}$,
\eqref{60} yields
\begin{equation}\label{10}
\begin{split}
\widehat{T}_{f_{+}}\widehat{T}_{g_{-}}-\widehat{T}_{g_{-}}\widehat{T}_{f_{+}}
=\widehat{T}_{g_{+}}\widehat{T}_{f_{-}}-\widehat{T}_{f_{-}}\widehat{T}_{g_{+}}.
\end{split}
\end{equation}

We need to find out the conditions for $f_{+},~f_{-},~g_{+},~g_{-}$ such that \eqref{10} holds.
In view of \eqref{20} and \eqref{49},
we distinguish four cases.

\noindent{\bf{Case 1.}}\,\,\,

(a)(2): $f_{+}=f_{+}(z_{1}),~g_{+}=g_{+}(z_{1}),~f_{-}=f_{-}(z_{2}),~g_{-}=g_{-}(z_{2}).$

(b)(1): $f_{+}=f_{+}(z_{2}),~g_{+}=g_{+}(z_{2}),~f_{-}=f_{-}(z_{1}),~g_{-}=g_{-}(z_{1}).$

According to Lemma \ref{3}, (a)(2) and (b)(1) ensure that \eqref{10} holds. Thus (I) and (II) would hold.

\noindent{\bf{Case 2.}}\,\,\,

(a)(1): $f_{+}=f_{+}(z_{1}),~g_{+}=g_{+}(z_{1}),~f_{-}=f_{-}(z_{1}),~g_{-}=g_{-}(z_{1}).$

(b)(2): $f_{+}=f_{+}(z_{2}),~g_{+}=g_{+}(z_{2}),~f_{-}=f_{-}(z_{2}),~g_{-}=g_{-}(z_{2}).$

One only  need consider case (a)(1), since (b)(2) is similar to (a)(1) .

In (a)(1),
taking the Berezin transform of both sides of  \eqref{10}, we have
\begin{equation*}
\begin{split}
\int_{\mathbb{T}^{2}}\big([\widehat{T}_{f_{+}},\widehat{T}_{g_{-}}]k_{{\lambda}_{1}}k_{{\lambda}_{2}}\big)
\overline{k_{{\lambda}_{1}}k_{{\lambda}_{2}}}d\sigma
&=\int_{\mathbb{T}^{2}}\big([\widehat{T}_{g_{+}},\widehat{T}_{f_{-}}]k_{{\lambda}_{1}}k_{{\lambda}_{2}}\big)
\overline{k_{{\lambda}_{1}}k_{{\lambda}_{2}}}d\sigma,\\
|{\lambda}_{2}|^{2}\big(f_{+}({\lambda}_{1})\overline{g_{-}({\lambda}_{1})}-\mathcal{P}[f_{+}g_{-}]({\lambda}_{1})\big)
&=|{\lambda}_{2}|^{2}\big(g_{+}({\lambda}_{1})\overline{f_{-}({\lambda}_{1})}-\mathcal{P}[g_{+}f_{-}]({\lambda}_{1})\big),\\
f_{+}({\lambda}_{1})\overline{g_{-}({\lambda}_{1})}-\mathcal{P}[{f_{+}g_{-}}]({\lambda}_{1})
&=g_{+}\small({\lambda}_{1}\small)\overline{f_{-}({\lambda}_{1})}-\mathcal{P}[g_{+}f_{-}]({\lambda}_{1})
~~\text{for}~{\lambda}_{2}\neq 0.
\end{split}
\end{equation*}

The second equality follows from Lemma \ref{1}, the last equation implies $f_{+}g_{-}-g_{+}f_{-}$
is harmonic function on $\mathbb{D}$, hence we get
\begin{equation}\label{11}
\begin{split}
\partial_{1}f_{+}\bar{\partial}_{1}g_{-}=\partial_{1}g_{+}\bar{\partial}_{1}f_{-}.
\end{split}
\end{equation}

We will use the method of \cite{axler1991commuting}  for finding all the solutions of \eqref{11}.

If $\partial_{1}g_{+}$ is identically 0 on $\mathbb{D},$ then \eqref{11} shows that either
$\partial_{1}f_{+}$ is identically 0 on $\mathbb{D}$ (so $f_{+}$ would be constant on $\mathbb{D}$
and (II) would hold)
or $\bar{\partial}_{1}g_{-}$ is identically 0 on $\mathbb{D}$ (so $g_{-}$ would be constant on $\mathbb{D}$
and (III) would hold). Similarly, if
$\bar{\partial}_{1}g_{-}$ is identically 0 on $\mathbb{D},$ then \eqref{11} shows that either
$\partial_{1}g_{+}$ is identically 0 on $\mathbb{D}$ (so $g_{+}$ would be constant on $\mathbb{D}$
and (III) would hold)
or $\bar{\partial}_{1}f_{-}$ is identically 0 on $\mathbb{D}$ (so $f_{-}$ be constant on $\mathbb{D}$
 and (I) would hold).
Thus we may assume that neither $\partial_{1}g_{+}$ nor $\bar{\partial}_{1}g_{-}$ is identically 0 on $\mathbb{D},$
and so \eqref{11} shows that
\begin{equation}
\begin{split}
\frac{\partial_{1}f_{+}}{\partial_{1}g_{+}}=\frac{\bar{\partial}_{1}f_{-}}{\bar{\partial}_{1}g_{-}}
\end{split}
\end{equation}
at all points of $\mathbb{D}$ except the countable set consisting of zeroes of $\partial_{1}g_{+}\bar{\partial}_{1}g_{-}.$
The left-hand side of the above equation is an analytic function (on $\mathbb{D}$ with the zeroes of $\partial_{1}g_{+}\bar{\partial}_{1}g_{-}$ deleted), and the right-hand side is the complex conjugate of an analytic function on the same domain, and so both sides must equal a constant
 $c\in\mathbb{D}.$ Thus $\partial_{1}f_{+}=c \partial_{1}g_{+}$ and $\bar{\partial}_{1}f_{-}=c \bar{\partial}_{1}g_{-}$
on $\mathbb{D}$. Hence $f_{+}-c g_{+}$ and $f_{-}-c g_{-}$ are constants on $\mathbb{D}$, and so their sum, which equals $f-cg$
is constant on $\mathbb{D}$, then (III) would hold.

\noindent{\bf{Case 3.}}\,\,\,

(c)(3): A nontrivial linear combination of $f_{+}$ and $g_{+}$ is constant and
a nontrivial linear combination of $f_{-}$ and $g_{-}$ is constant.

We may assume, without loss of generality, that $f_{+}=\alpha g_{+}+\lambda$ and $f_{-}={\beta}{g_{-}}+{\mu},$
where $\alpha,~\lambda,~\beta$ and  $\mu $ are constants. By substituting $f_{+}=\alpha g_{+}+\lambda$ and
$f_{-}={\beta}{g_{-}}+{\mu}$ into \eqref{10}, we obtain
\[(\alpha-\beta)(\widehat{T}_{g_{+}}\widehat{T}_{g_{-}}-\widehat{T}_{g_{-}}\widehat{T}_{g_{+}})=0.\]

If $\alpha-\beta=0$, then $f-\alpha g$ is constant, so (III) would hold. If $\widehat{T}_{g_{+}}\widehat{T}_{g_{-}}-\widehat{T}_{g_{-}}\widehat{T}_{g_{+}}=0$,
using Theorem \ref{3}, we have
\begin{equation}
\begin{split}
\rm \left\{ \begin{array}{lll}
     \ (X)~ g_{+}=g_{+}(z_{1}),g_{-}=g_{-}(z_{2});\text{or}  \\
      \ (Y)~ g_{+}=g_{+}(z_{2}),g_{-}=g_{-}(z_{1});\text{or} \\
       \ (Z)~ g_{+}~\text{or}~g_{-}~\text{is~constant}.
     \end{array} \right.
\end{split}
\end{equation}
Hence (X) implies (I), (Y) implies (II) and (Z) implies (III).

\noindent{\bf{Case 4.}}\,\,\,

(a)(3):  $f_{+}=f_{+}(z_{1}),~g_{+}=g_{+}(z_{1}),$ a nontrivial linear combination of $f_{-}$ and $g_{-}$ is constant.

(b)(3):  $f_{+}=f_{+}(z_{2}),~g_{+}=g_{+}(z_{2}),$ a nontrivial linear combination of $f_{-}$ and $g_{-}$ is constant.

(c)(1):  a nontrivial linear combination of $f_{+}$ and $g_{+}$ is constant, $f_{-}=f_{-}(z_{1}),~g_{-}=g_{-}(z_{1}).$

(c)(2):  a nontrivial linear combination of $f_{+}$ and $g_{+}$ is constant, $f_{-}=f_{-}(z_{2}),~g_{-}=g_{-}(z_{2}).$

Since (a)(3), (b)(3), (c)(1) and (c)(2) are similar, one only  need consider case (a)(3).
For (a)(3), without loss of generality, we may assume that there exist constants $\delta$ and $\eta $ such that $f_{-}=\delta g_{-}+\eta .$
Therefore, \eqref{10} entails that
\begin{equation*}
\begin{split}
\widehat{T}_{f_{+}}\widehat{T}_{g_{-}}-\widehat{T}_{g_{-}}\widehat{T}_{f_{+}}
&=\widehat{T}_{g_{+}}\widehat{T}_{f_{-}}-\widehat{T}_{f_{-}}\widehat{T}_{g_{+}}\\
&=\widehat{T}_{g_{+}}(\delta \widehat{T}_{g_{-}} + \eta I)-(\delta \widehat{T}_{g_{-}} + \eta I)\widehat{T}_{g_{+}}\\
&=\delta(\widehat{T}_{g_{+}}\widehat{T}_{g_{-}}-\widehat{T}_{g_{-}}\widehat{T}_{g_{+}}),
\end{split}
\end{equation*}
it follows that
\begin{equation}\label{61}
\begin{split}
\widehat{T}_{f_{+}-\delta g_{+}}\widehat{T}_{g_{-}}
=\widehat{T}_{g_{-}}\widehat{T}_{f_{+}-\delta g_{+}}.
\end{split}
\end{equation}

We can apply Lemma \ref{3} to \eqref{61}, if $f_{+}-\delta g_{+}$ is constant, then (III) would hold,
if $g_{-}$ is constant or $g_{-}=g_{-}(z_{2}),$ then (I) would hold.

Conversely, it is easy to show that a nontrivial linear combination of $f$ and $g$ is constant implies  $\widehat{T}_{f}\widehat{T}_{g}
=\widehat{T}_{g}\widehat{T}_{f}$. Since $$ \widehat{T}_{f}\widehat{T}_{g}-\widehat{T}_{g}\widehat{T}_{f}
=\widehat{T}_{f}\widehat{T}_{g}-\widehat{T}_{fg}+\widehat{T}_{fg}-\widehat{T}_{g}\widehat{T}_{f},$$
by Theorem \ref{27},
we obtain that (I) and (II) are true.
\end{proof}

Theorem \ref{26} has the following consequence.
\begin{corollary}
Let $f\in h^{2}(\mathbb{T}^{2})\bigcap L^{\infty}(\mathbb{T}^{2})$, then $\widehat{T}_f$ is normal if and only if
the range of $f$ lies on a line.
\end{corollary}

\begin{proof} The Toeplitz operator $\widehat{T}_f$ is normal if and only if
 $\widehat{T}_f$ and $\widehat{T}_{\bar{f}}$ commute. By Theorem \ref{26} this is the case if
 and only if there are constants $\alpha$ and $\beta$, not both zero, such that
 $\alpha f+\beta \bar{f}$ is constant.
\end{proof}
{\bf Acknowledgments.} The authors are grateful to Professor Dechao
Zheng for various helpful suggestions. The second author is supported by the National Natural Science Foundation of China(11271388)
and the Program for University Innovation Team of Chongqing (CXTDX201601026).

\bibliographystyle{amsplain}

\end{document}